\documentclass[11pt]{article}
\usepackage[letterpaper]{geometry}
\pdfoutput=1
\usepackage{amsmath,amsthm,amsfonts,amssymb}
\usepackage{enumerate,color,xcolor}
\usepackage{graphicx}
\usepackage{subfigure}
\usepackage{comment}

\usepackage{booktabs}

\usepackage{url}

\numberwithin{equation}{section}
\theoremstyle{plain}

\newtheorem{theorem}{Theorem}

\newtheorem{corollary}[theorem]{Corollary}

\newtheorem{lemma}[theorem]{Lemma}

\newtheorem{proposition}[theorem]{Proposition}

\newtheorem{remark}[theorem]{Remark}
\newtheorem{assumption}[theorem]{Assumption}
\newcommand{\red}[1]{\textcolor{black}{#1}}

\newcommand{\rev}[1]{\textcolor{blue}{#1}}
\usepackage{booktabs,tabularx}
\usepackage{multirow}

\begin{document}

\begin{center}
  \Large \bf Functional central limit theorems for stationary Hawkes processes and application to infinite--server queues
\end{center}

\author{}
\begin{center}
{Xuefeng
  Gao}\,\footnote{Corresponding Author. Department of Systems
    Engineering and Engineering Management, The Chinese University of Hong Kong, Shatin, N.T. Hong Kong;
    xfgao@se.cuhk.edu.hk},
  Lingjiong Zhu\,\footnote{Department of Mathematics, Florida State University, 1017 Academic Way, Tallahassee, FL-32306, United States of America; zhu@math.fsu.edu.
  }
\end{center}

\begin{center}
 \today
\end{center}

\begin{abstract}
A univariate Hawkes process is a simple point process that is self-exciting and has
clustering effect. The intensity of this point process is given by the sum of a baseline intensity and another term that depends on the entire past history of the point process. Hawkes process has wide applications in finance, neuroscience, social networks, criminology, seismology,
and many other fields. In this paper, we prove a functional central limit theorem for stationary Hawkes processes in the asymptotic regime where the baseline intensity is large. The limit is a non-Markovian Gaussian process with dependent increments. We use the resulting approximation to study an infinite-server queue with high-volume Hawkes traffic. We show that the queue length process can be approximated by a Gaussian process, for which we compute explicitly the covariance function and the steady-state distribution. We also extend our results to multivariate stationary Hawkes processes and establish limit theorems for infinite-server queues with multivariate Hawkes traffic.
\end{abstract}

\section{Introduction}
A univariate linear Hawkes process is a simple point process $N$ whose (stochastic) intensity $\lambda$ at time $t$ is given by
\begin{equation*}
\lambda({t}):= \mu + \int_{-\infty}^{t-}h(t-s)N(ds) = \mu + \sum_{\tau_i<t}h(t-\tau_i),  \label{dynamics}
\end{equation*}
where
$\tau_i$ are the occurrences of the points before time $t$, and $h(\cdot):[0, \infty) \rightarrow [0, \infty)$.
See Section~\ref{sec:hawkes} for accurate definitions, multivariate extensions, and related properties.
We use the notation $N({t}):=N(0,t]$ to denote the number of
points in the interval $(0,t]$. When $h \equiv 0$, the Hawkes process $N$ becomes a Poisson process with rate $\mu$.
In the literature, the parameter $\mu$ is called the \textit{baseline intensity}, and $h(\cdot)$ is called
the \textit{exciting function} or sometimes referred to as the \textit{kernel function}.

The linear Hawkes process was first introduced by A.G. Hawkes in 1971 \cite{Hawkes, Hawkes71II}. It exhibits both self--exciting (i.e., the occurrence
of an event increases the probabilities of future
events) and clustering properties. Hence it is very appealing in point process modeling and it has wide applications
in various domains, including neuroscience \cite{Johnson96, Pernice2012, Reynaud2013}, seismology \cite{Ogata1988}, genome analysis \cite{Gusto2005, Reynaud2010}, social network \cite{Blundell2012, Crane2008}, finance (see the recent survey paper \cite{Bacry2015} and the references therein) and others.

This paper focuses on stationary Hawkes processes and their applications in specific queueing systems. A Hawkes process is stationary if its distribution does not change under time shift. See Section~\ref{sec:hawkes} for accurate definitions. In this paper, we develop approximations of a stationary Hawkes process with a large baseline intensity $\mu$. Mathematically, under a mild assumption on the exciting function (Assumption~1), we
establish a functional central limit theorem (FCLT) for a sequence of univariate stationary Hawkes processes $N^{\mu}$ indexed by the baseline intensity $\mu$ which goes to infinity (see Theorem~\ref{thm:FCLT}). These Hawkes processes share a common fixed exciting function $h$. The limit process turns out to be a Gaussian process which is non--Markovian unless $h \equiv 0$. This limiting Gaussian process has stationary but dependent increments.

To illustrate the strength of the Gaussian approximation for the Hawkes processes, we study a specific queueing model with a stationary Hawkes traffic: an infinite--server queue with general service time distributions. The Hawkes process could be a potential traffic model especially for financial market data feeds\footnote{Market data feeds are typically composed of event messages that provide, in real time, the status of the market such as asset prices, reports of completed trades, and order activities. While some industry white paper \cite{Low} suggests that the market data traffic clearly exhibits clustering, we are not aware of academic studies or publicly available data on market data feeds.} for several reasons. First,
the Hawkes process
naturally extends the classical Poisson process. Second, stock order flows and the occurrence of financial market events are known to exhibit clustering (in time) and self--exciting features, e.g., trades trigger other trades \cite{Bacry2015, Bowsher2007, Cont2012, Fonseca2014, Hewlett2006}. The standard Poisson process
can not capture these features while the Hawkes process can adequately model such clustering and self--exciting behavior.
Third, the Hawkes process is a highly versatile and flexible model which can exhibit a broad range of correlation structure, depending on the specification of how past events affect the occurrence of current and future events (\cite{Bacry2015}). Finally, the Hawkes process is amenable to statistical inference (see, e.g., \cite{Bacry2015, Daley, Ozaki1979}).

Infinite--server queues are interesting in their own right since they naturally arise in the study of many applications such as electric power consumption and insurance mathematics \cite{Blanchet2014, Glynn2002}. In addition, as argued in \cite{whittIS2012}, infinite--server queues often serve as useful approximations
for multi--server queues which are classical models for large--scale service systems (e.g., server farms, call centers). In the financial context, such an infinite--server queue can serve as an approximate model for describing the market data feed sent from exchanges, processed by many parallel computer servers, and then delivered to consuming applications of end--users.

Since Hawkes processes are non--Markovian in general \red{and the inter-arrival times are correlated, it is challenging to analyze the performance of an infinite-server queue with Hawkes traffic and general service time distributions, either analytically or numerically.} Hence, we consider the regime that the baseline intensity $\mu$ of the Hawkes input is large. Such a regime could be relevant since the market data traffic is generated by many market participants and the market
data volumes are huge in practice (e.g., in the range of gigabytes per second). Relying on \cite{Kri1997}, we develop heavy--traffic approximations for the performance of such an infinite--server queue fed by a univariate stationary Hawkes process with a large baseline intensity $\mu$ (Proposition~\ref{thm:IS1}). The limiting queue length process is a Gaussian process.  We compute its covariance function as well as its steady--state distribution explicitly, both of which depend on the distribution of service times as well as the detailed form of the covariance density of the Hawkes traffic (Proposition~\ref{prop:covX} and Corollary~\ref{prop:X-infty}). In the special case of exponential service time distributions, the limiting queue length process is an Ornstein-Uhlenbeck (OU) process driven by a Gaussian process. This Gaussian--driven OU process is non--Markovian in general. We illustrate through examples and numerical experiments that the Gaussian approximation for the steady--state queue length is effective.


We also extend our functional central limit theorem to multivariate stationary Hawkes processes (Theorem~\ref{thm:multi-FCLT}) and study infinite--server queues with multivariate Hawkes traffic and general service time distributions. Such a model can be viewed as a multi--class queueing model with correlated and mutually--exciting arrivals. We show that the limiting queue length process is a multivariate Gaussian process (Proposition~\ref{thm:IS-multi-gen}).
When the service times of each class of customers are independent exponentials, this limiting queue length process becomes a multi--dimensional Gaussian--driven OU process (Proposition~\ref{thm:IS-mutual}).

To summarize, our paper is the
first one that studies the large baseline intensity asymptotics for stationary Hawkes processes. Unlike the existing limit theorems for Hawkes process in the
literature, our proof relies on the immigration-birth representation of the
linear Hawkes processes \cite{HawkesII}, and the delicate analysis of the moments of the stationary Hawkes process. Our paper is also the first to study queues with stationary Hawkes traffic. We obtain new explicit results for the performance of infinite-server queueing systems which allows us to better understand the impact of self--exciting and mutually--exciting Hawkes traffic on the system performance.

\bigskip
\textbf{Related Literature.}
Two streams of research that are closely related to our work are Hawkes processes and infinite-server queues.
We now explain the difference between our study and the existing literature in these two areas.

\textit{Asymptotics of Hawkes processes.}
Note that most of the existing literature on limit theorems
for Hawkes processes are for {\it{large--time asymptotics}}, where one scales both time and space. See \cite{Bacry, Bordenave, ZhuMDP} for large--time asymptotics of linear Hawkes processes,  \cite{Karabash, ZhuCIR} for large--time asymptotics for extensions of linear Hawkes processes,
\cite{Jaisson, JaissonII} for the nearly unstable case where $\Vert h\Vert_{L^{1}}\approx 1$, \cite{Zhang2015} for the generalized {Markovian} Hawkes processes (or affine point processes), and \cite{ZhuThesis} for large--time asymptotics of nonlinear Hawkes processes.


These large--time asymptotics are different from our large-$\mu$ asymptotics (no time-scaling is involved).
We will see later, see e.g. Theorem \ref{thm:FCLT},
that the time-space and intensity-space scalings are not equivalent.
For Poisson processes, these two scalings are equivalent and both lead to a Brownian limit.
For Hawkes processes, for the time-space scaling, we obtain the Brownian limit, see e.g. \cite{ZhuCLT}.
On the other hand, if we consider large baseline intensity $\mu$ and scale down the space, we get a non-Markovian Gaussian limit (Theorem \ref{thm:FCLT}). The primary reason is that the Hawkes process $N^{\mu}$ with a baseline intensity $\mu$, say $\mu$ is a positive integer,
can be expressed as partial sums of i.i.d. copies of a Hawkes process $N^{1}$ which has baseline intensity one (see Sections~2 and 3). Thus for the intensity-space scaling we consider,
the covariance structure of $N^1$ is still preserved in the limit, and
the covariance structure of $N^1$ does not coincide with that of a Brownian motion
since $N^1$ has dependent time increments, which leads to the non-Brownian Gaussian limit.

Other than the large--time asymptotics, limit theorems for non-stationary Markovian Hawkes processes with a large initial intensity have been established in our recent studies \cite{GZ, GZ2}.
Large--dimension asymptotics have been studied in \cite{Delattre,Chevallier,DF}, in which the authors studied the asymptotics
for the multivariate Hawkes process and its extensions
where the number of dimension goes to infinity, and obtained a mean--field limit.

\textit{Infinite-server queues.}
In the setting of infinite--server queues, our work complements the stream of research on heavy--traffic approximations of such queues, see, e.g., \cite{Eick93, Iglehart1965, Jamol2016, Pang2015, Pang2007, Pang2010, Reed2015, whitt2002} and the references therein. In these studies, the heavy--traffic limit of the arrival process is typically a Brownian motion or a deterministic time--changed Brownian motion. With Hawkes traffic, we obtain a non--Markovian limit but the Gaussian structure still allows us to obtain elegant formulas for transient and steady--state performance measures. From the traffic modeling perspective, we also mention that certain Poisson cluster processes have been used to model tele--traffic arrivals (see, e.g., \cite{Fasen2010, Fay2006, Hohn2003}).
The linear Hawkes processes which can be seen as Poisson cluster processes (see e.g. \cite{Bacry2015, Daley}) are not covered by these studies.

\bigskip
\textbf{Organization of this paper.}
The rest of the paper is organized as follows. In Section~\ref{sec:hawkes},
we formally introduce stationary linear Hawkes processes and review some of their properties.
In Section~\ref{sec:2}, we state the main result on the functional central limit theorem for univariate stationary Hawkes processes with large baseline intensity $\mu$ and describe the properties of the limiting Gaussian process.
In Section~\ref{sec:inf-server}, we develop heavy--traffic approximations for infinite--server queues with univariate Hawkes traffic. We also discuss in detail the special case when service times are exponentially distributed. In Section~\ref{sec:multi-hawkes}, we extend our results to multivariate stationary Hawkes processes and study infinite--server queues with multivariate Hawkes traffic.
The proofs of all the results are collected in the Appendix.

\section{Introduction to Stationary Hawkes processes} \label{sec:hawkes}
In this section, we formally introduce stationary linear Hawkes processes and review some of their properties.

\subsection{Definition and stationarity condition}

Let $N$ be a simple point process on $\mathbb{R}$, that is,
a family $\{N(C)\}_{C\in\mathcal{B}(\mathbb{R})}$ of random variables
with values in $\{0, 1, 2, \ldots, \}\cup\{\infty\}$ indexed
by the Borel $\sigma$-algebra $\mathcal{B}(\mathbb{R})$ of the real line $\mathbb{R}$,
where $N(C)=\sum_{n\in\mathbb{Z}}1_{C}(T_{n})$ and $(T_{n})_{n\in\mathbb{Z}}$
is a sequence of extended real-valued random variables so that almost surely $T_{0}\leq 0<T_{1}$, $T_{n}<T_{n+1}$
on $\{T_{n}<\infty\}\cap\{T_{n+1}>-\infty\}$ for every $n\in\mathbb{Z}$.
Let $\mathcal{F}_{t}=\sigma(N(C),C\in\mathcal{B}(\mathbb{R}),C\subset(-\infty,t]))$.
The process $\lambda(t)$ is called the $\mathcal{F}_{t}$-intensity of $N$ if for
all intervals $(a,b]$, we have
\begin{equation}\label{martingality}
\mathbb{E}[N((a,b])|\mathcal{F}_{a}]
=\mathbb{E}\left[\int_{a}^{b}\lambda(s)ds\Big|\mathcal{F}_{a}\right],
\qquad
\text{a.s.}
\end{equation}

The univariate linear Hawkes process with baseline intensity $\mu>0$ and exciting function $h:\mathbb{R}_{+}\rightarrow\mathbb{R}_{+}$
is a simple point process $N$ admitting the $\mathcal{F}_{t}$-intensity
\begin{equation}\label{intensity}
\lambda(t)=\mu+\int_{-\infty}^{t-}h(t-s)N(ds).
\end{equation}
Due to \eqref{intensity}, the univariate Hawkes process is sometimes also called the self--exciting point process in the literature.

A commonly used nontrivial example of the exciting function $h$ is an exponential function, i.e.,
$h(t)= \alpha e^{-\beta t}$ for $t \ge 0$, where $\alpha, \beta >0$. In this special case, the process $(\lambda(t), N(t))$
is Markovian, and the intensity process $\lambda(t)$ itself is also Markovian, see e.g. \cite{Errais}.
The power law function $h(t)=\frac{1}{(1+\delta t)^{\gamma}}$, where $\delta,\gamma>0$ is also
a popular choice for the exciting function in the literature, see e.g. \cite{Bacry2015}.

The multivariate Hawkes process extends the univariate Hawkes process to $k \ge 1$ dimensions as follows.
Let $\mathbb{N}:=(\mathbb{N}^{1} ,\ldots, \mathbb{N}^{k})$,
where $\mathbb{N}^{i}$ are simple point processes on $\mathbb{R}$ with no common points, and
for each $1\leq i\leq k$, $\mathbb{N}^{i}$ has the intensity:
\begin{equation}\label{TakeExpectation}
\lambda^{i}(t)=\mu_{i}+\sum_{j=1}^{k}\int_{-\infty}^{t-}h_{ij}(t-s) \mathbb{N}^{j}(ds),
\end{equation}
where $\mu_{i}>0$ and $h_{ij}(\cdot):\mathbb{R}_{+}\rightarrow\mathbb{R}_{+}$ for $1\leq i, j\leq k$. Due to \eqref{TakeExpectation}, the multivariate Hawkes process is sometimes also called the mutually--exciting point process in the literature.

To facilitate the presentation, we summarize below the key properties of the linear stationary Hawkes processes that will be used in the paper. Write $\Vert f\Vert_{L^{1}} = \int_{0}^{\infty} f(t)dt$
for a function $f:[0,\infty)\rightarrow[0,\infty)$.
\begin{enumerate}
\item [(a)]
(Stationarity).
For a simple point process $N$, stationarity of $N$ means its distribution does not change under time shift. More precisely, $N$ is stationary if
 the process $\theta_{t}N$ has the same distribution as the process $N$ for any $t$, where
$\theta_{t}$ is a shift operator defined as $\theta_{t}N(C)=N(t+C)$ for every $C\in\mathcal{B}(\mathbb{R})$.
This directly implies that a stationary Hawkes process $N$ has stationary increments, and
the intensity process $\lambda(\cdot)$ is a stationary process where the distribution of $\lambda(t)$ does not depend on $t.$
Similarly, we say a multivariate point process $\mathbb{N}=(\mathbb{N}^{1} ,\ldots, \mathbb{N}^{k})$ is stationary, if
$(\theta_{t} \mathbb{N}^{1},\ldots,\theta_{t} \mathbb{N}^{k})$ has the same distribution as $(\mathbb{N}^{1} ,\ldots, \mathbb{N}^{k})$ for any $t$.

Under the assumption $\Vert h\Vert_{L^{1}}<1$, there is a unique stationary
version of the Hawkes process $N$ with the intensity \eqref{intensity}, see e.g. \cite{Bremaud}.
More generally, under the assumption that the spectral radius of the $k \times k$ matrix
$\mathbb{H}:= (\Vert h_{ij}\Vert_{L^{1}})_{1\leq i,j\leq k}$ is strictly less than $1$,
there is a unique stationary version of the multivariate Hawkes process $\mathbb{N}$
with the intensity \eqref{TakeExpectation}, see e.g. \cite{Bremaud}.

\item [(b)]
(Martingality). By the definition of the intensity in \eqref{martingality}, we have
for any simple point process $N$ with the intensity $\lambda$,
$N(t)-\int_{0}^{t}\lambda(s)ds$ is a martingale.
Moreover, its predictable quadratic variation is given by $\int_{0}^{t}\lambda(s)ds$
so that $\left(N(t)-\int_{0}^{t}\lambda(s)ds\right)^{2}-\int_{0}^{t}\lambda(s)ds$ is also
a martingale. We will apply this martingale property to univariate stationary Hawkes processes and the marginal processes of multivariate stationary Hawkes processes in the proofs of Theorem~\ref{thm:FCLT} and \ref{thm:multi-FCLT}.

\item [(c)]
(First-order mean). For stationary $k-$variate Hawkes processes,
by taking expectations on both hand sides of \eqref{TakeExpectation}
and by the martingale property \eqref{martingality}, we have for each $t,$
\begin{equation*}
\bar{\lambda}_{i}:=\mathbb{E}[\lambda^{i}(t)]
=\mu_{i}+\sum_{j=1}^{k}\int_{-\infty}^{t-}h_{ij}(t-s)\bar{\lambda}_{j}ds,
\end{equation*}
which implies that
\begin{equation}\label{BarLambdaDefn}
\bar{\lambda}=(\mathbb{I}-\mathbb{H})^{-1}\mu,
\end{equation}
where $\bar{\lambda}=(\bar{\lambda}_{i})_{1\leq i\leq k}$, $\mu=(\mu_{i})_{1\leq i\leq k}$ and $\mathbb{I}$ is the identity matrix.

\item [(d)]
(Covariance density and variance function). For a stationary $k-$variate Hawkes process $(\mathbb{N}^{1} ,\ldots, \mathbb{N}^{k})$,
the covariance density matrix $\Phi(\tau)=(\Phi_{ij}(\tau))_{1\leq i,j\leq k}$, where $\Phi_{ij}(\tau) := \mathbb{E} [d\mathbb{N}^i(t+\tau) d\mathbb{N}^j(t)]/(dt)^2 - \bar \lambda_i \bar \lambda_j$ which does not depend on $t$, is given as follows,
see e.g. \cite{Hawkes,Hawkes71II}.
For $\tau\geq 0$,
\begin{equation} \label{eq:cov-density}
\Phi(\tau)=h(\tau)\text{diag}(\bar{\lambda})+\int_{-\infty}^{\tau}h(\tau-v)\Phi(v)dv,
\end{equation}
and $\Phi_{ij}(-\tau)=\Phi_{ji}(\tau)$ for every $\tau>0$ and $1\leq i,j\leq k$,
where $\bar{\lambda}$ is defined in \eqref{BarLambdaDefn}.
Here $\text{diag}(\bar{\lambda})$ is the diagonal matrix with entries $\bar{\lambda}_{i}$'s on the diagonal,
and with slight abuse of notations, $h(t)=(h_{ij}(t))_{1\leq i,j\leq k}$.
The variance function for the stationary $k-$variate Hawkes process, $\mathbb{K}(t)=(K_{ij}(t))_{1\leq i,j\leq k}
:=\text{Var}(\mathbb{N}(t))=(\text{Cov}(\mathbb{N}^{i}(t),\mathbb{N}^{j}(t)))_{1\leq i,j\leq k}$ is given by
\begin{equation}\label{eq:var-function}
\mathbb{K}(t):=\text{diag}(\bar{\lambda})t+2\int_{0}^{t}\int_{0}^{t_{2}}\Phi(t_{2}-t_{1})dt_{1}dt_{2}.
\end{equation}

\item [(e)]
(Association).
Intuitively, since the Hawkes process has the self- and mutually-exciting properties,
there are positive correlations between counts across time intervals. To make this statement rigorous, 
we will use the notion of \emph{association} from probability theory, see e.g. \cite{Evans}. 
Let $\mathcal{X}$ and $\mathcal{Y}$ be complete and separable metric spaces,
with closed orders $\leq_{\mathcal{X}}$ and $\leq_{\mathcal{Y}}$. A map $f:\mathcal{X}\rightarrow\mathcal{Y}$
is non-decreasing if $x_{1}\leq_{\mathcal{X}}x_{2}$ implies $f(x_{1})\leq_{\mathcal{Y}}f(x_{2})$.
An $\mathcal{X}$-valued random variable $X$ is \emph{associated} if for each
pair of bounded, Borel measurable, non-decreasing functions $f,g:\mathcal{X}\rightarrow\mathbb{R}$,
we have $\text{Cov}(f(X),g(X))\geq 0$. 
Let $\mathcal{S}$ be a locally compact, separable, metric space
and denote by $M(\mathcal{S})$ the space of Radon measures on $\mathcal{S}$
equipped with the vague topology with a partial ordering which is closed
by declaring that $\mu\leq\nu$ if $\mu(B)\leq\nu(B)$ for all Borel sets $B$.
A random measure is an $M(\mathcal{S})$-valued random variable.
The linear $k-$variate Hawkes process $\mathbb{N}=(\mathbb{N}^{1},\ldots,\mathbb{N}^{k})$
is equivalent to a marked linear Hawkes process $\mathbb{N}^{\dagger}$,
which is a random measure defined
on the space $\mathcal{S}=\mathbb{R}\times\{1,2,\ldots,k\}$, 
via $\mathbb{N}^{\dagger}(C,i)=\mathbb{N}^{i}(C)$, for any Borel sets $C$ of $\mathbb{R}$
and $i\in\{1,2,\ldots,k\}$. The random measure $\mathbb{N}^{\dagger}$
is infinitely divisible since the $k-$variate linear Hawkes process $\mathbb{N}$ is a special case of the Poisson cluster process (see e.g. \cite{Daley, JHR}), which is infinitely divisible.
Theorem 1.1. in \cite{Evans} (which first appears in \cite{BW}) 
says any infinitely divisible random measure on $\mathcal{S}$ is associated.
The association property of Hawkes processes implies that
the covariance density \eqref{eq:cov-density} is non-negative and it will also
be used to show the finiteness of the moment generating function of
the stationary Hawkes process in the proof of Theorem~\ref{thm:multi-FCLT}.
\end{enumerate}

Throughout the paper, we will always assume that we are working with the stationary
version of a Hawkes process. More specifically, we will
make the following assumption on the exciting function of $k-$dimensional Hawkes processes which guarantees the existence of the stationary version. This assumption is satisfied in most applications of Hawkes processes, see e.g. \cite{Bacry2015, Hawkes, ZhuThesis} and the references therein.

\begin{assumption}\label{assump1}
For all $1\leq i,j\leq k$, the exciting function $h_{ij}$ is non-negative, locally bounded, and Riemann integrable.
In addition, the spectral radius of the $k \times k$ matrix
$\mathbb{H}:= (\Vert h_{ij}\Vert_{L^{1}})_{1\leq i,j\leq k}$ is strictly less than $1$.
\end{assumption}

\subsection{Immigration--birth representation}\label{ImmigrationSection}

In this section, we review the well--known immigration birth representation of linear Hawkes processes (see, e.g., \cite{HawkesII, JHR}) which is the key to the proof of our results.

For the univariate stationary Hawkes process with intensity dynamics \eqref{intensity},
we assume that immigrants arrive according to a homogeneous
Poisson process with constant rate $\mu$ on the real line $\mathbb{R}$.
Each immigrant would produce children and the number
of children has a Poisson distribution with mean $\Vert h\Vert_{L^{1}}$. Conditional on the number of
the children of an immigrant, the children are born independently, and each child is born
at a time with a probability density function $\frac{h(t)}{\Vert h\Vert_{L^{1}}}$.
In other words, children are born according to an inhomogeneous Poisson process
with intensity $h(\cdot)$. Each child would produce children
according to the same laws independent of other children. All the immigrants produce children independently.
The number of points of a linear Hawkes process on a time interval $(0,t]$ equals
the total number of immigrants and the descendants on the interval $(0,t]$.

Note that the immigration--birth representation holds similarly for the multivariate Hawkes process, see e.g. \cite{JHR}.
For a $k$--variate Hawkes process $(\mathbb{N}^{1},\ldots,\mathbb{N}^{k})$
with the intensity \eqref{TakeExpectation} for $\mathbb{N}^{i}$, where $1\leq i\leq k$,
we consider immigrants of $k$ types, and the type-$i$ immigrants arrive
according to a homogeneous Poisson process with intensity $\mu_{i}$,
and each type-$i$ immigrant produce children of type $j$ according to
an inhomogeneous Poisson process with intensity $h_{ji}(\cdot)$.
Each child of type $i$ would produce children of different types according
to the same laws independent of other children.
All the immigrants produce children independently.
The number of points $\mathbb{N}^{i}$ on a time interval $(0,t]$
equals to the total number of immigrants and the  descendants of type $i$ on the interval $(0,t]$.

Also note that the immigration--birth representation does not require the stationarity of the Hawkes process,
or the monotonicity of the exciting function, see e.g. \cite{ZhuMDP}.

\section{FCLT for univariate stationary Hawkes processes} \label{sec:2}
In this section we develop approximations for a \textit{univariate} stationary Hawkes process with a large baseline intensity $\mu$.

Consider a univariate stationary Hawkes process $N^{\mu}$ with stochastic intensity in \eqref{intensity}. We write $N^{\mu}$
to emphasize that the baseline intensity of this Hawkes process is $\mu$. Our goal is to
establish a functional central limit theorem for a sequence of stationary Hawkes processes $N^{\mu}$ in the asymptotic regime $\mu \rightarrow \infty$.
Note that the exciting function is fixed, i.e., this sequence of Hawkes processes shares a common exciting function $h$ with $\Vert h\Vert_{L^{1}} <1$.

To facilitate the presentation, let us define
\begin{equation} \label{eq:var}
K(t):=\frac{t}{1-\Vert h\Vert_{L^{1}}}
+2\int_0^t \int_0^{t_2}\phi(t_{2}-t_{1})dt_{1}dt_{2},
\end{equation}
where $\phi:[0, \infty)\rightarrow [0, \infty)$ satisfies the integral equation:
\begin{equation} \label{eq:phi}
\phi(t)=\frac{h(t)}{1-\Vert h\Vert_{L^{1}}}+\int_{0}^{\infty}h(t+v)\phi(v)dv
+\int_{0}^{t}h(t-v)\phi(v)dv,
\end{equation}
and $\phi(-t) = \phi(t)$ for $t>0$.
The function $\phi$ and $K$ are just the covariance density and variance functions
for the univariate stationary Hawkes process with baseline intensity $1$, respectively. See Equations~\eqref{eq:cov-density} and \eqref{eq:var-function}.
Note that the covariance density $\phi$ is non-negative
since the linear Hawkes process is associated.
When $h\equiv 0$, the linear Hawkes process reduces to the Poisson process with independent increments
and thus $\phi\equiv 0$. On the other hand, when $\phi\equiv 0$, from \eqref{eq:phi},
it is clear that $h\equiv 0$. Hence, $\phi\equiv 0$ if and only if $h\equiv 0$.

We now present a result on the functional central limit theorem for such univariate stationary Hawkes processes. Write $(D([0,\infty),\mathbb{R}),J_{1})$ as the space of c\`{a}dl\`{a}g processes on $[0, \infty)$ that are equipped with Skorohod $J_1$ topology (see, e.g., Billingsley \cite{Billingsley}), and write $``\Rightarrow"$ for convergence in distribution. Recall from \eqref{BarLambdaDefn} that $\bar \lambda = \frac{\mu}{1 - \Vert h\Vert_{L^{1}}}.$
\begin{theorem} \label{thm:FCLT}
Under Assumption~\ref{assump1}, we have as $\mu\rightarrow\infty$,
\begin{equation*}
\frac{N^{\mu}(t)- \bar \lambda t}{\sqrt{\mu}}\Rightarrow G(t),
\end{equation*}
in $(D([0,\infty),\mathbb{R}),J_{1})$, where $G$ is a mean-zero almost surely continuous Gaussian process
with the covariance function, for $t \ge s$,
\begin{equation}\label{eq:cov-G}
\mbox{Cov}(G(t),G(s))=\int_{s}^{t}\int_{0}^{s}\phi(u-v)dvdu+K(s).
\end{equation}
\end{theorem}

The proof of this result is given in Appendix~\ref{sec:proof1}.

We now briefly explain the intuition behind this result.
Without loss of generality, we assume $\mu$ takes integer values. By the immigration--birth representation of Hawkes processes, one can deduce that for a stationary univariate Hawkes process $N^{\mu}$ with a baseline intensity $\mu$ and an exciting function $h$, we can decompose it as the sum of $\mu$ i.i.d stationary Hawkes processes, each having a baseline intensity one and an exciting function $h$. Then one expecte by central limit theorem type of arguments that $N^{\mu}$ will be asymptotically Gaussian when we send $\mu$ to infinity.

We next discuss the covariance function of $G$ in \eqref{eq:cov-G}. 
In general, the covariance function of $G$ in \eqref{eq:cov-G} is semi-explicit and we can compute it by first numerically solving $\phi$ via the integral equation \eqref{eq:phi}.
In the special case when $h(t)= \alpha e^{-\beta t}$ where $\alpha< \beta$, the covariance function of $G$ is explicit. To see this, we first deduce from \eqref{eq:phi} that
\begin{equation*}
\phi(t)=\frac{\alpha e^{-\beta t}}{1- \frac{\alpha}{\beta}}+ \alpha e^{-\beta t} \cdot \int_{0}^{\infty}e^{-\beta v}\phi(v)dv
+\alpha e^{-\beta t} \cdot \int_{0}^{t} e^{\beta v}\phi(v)dv,
\end{equation*}
which yields that
\begin{equation}\label{eq:phi-exp}
\phi(t)=\frac{\alpha\beta(2\beta-\alpha)}{2(\beta-\alpha)^{2}}e^{-(\beta-\alpha)t},\qquad t\ge 0.
\end{equation}
Plugging this into \eqref{eq:var}, we find that
\begin{align}\label{eq:Kt-exp}
\mbox{Var}(G(t)) = K(t)&=\frac{t}{1-\frac{\alpha}{\beta}}
+2\frac{\alpha\beta(2\beta-\alpha)}{2(\beta-\alpha)^{2}}\int_{0}^{t}\int_{0}^{t_{2}}e^{-(\beta-\alpha)(t_{2}-t_{1})}dt_{1}dt_{2}
\nonumber\\
&=\frac{\beta^3 }{(\beta-\alpha)^3} t  -\frac{\alpha\beta(2\beta-\alpha)}{(\beta-\alpha)^{4}}\left[1-e^{-(\beta-\alpha)t}\right],
\end{align}
and for $t \ge s$,
\begin{eqnarray*}\label{eq:cov-G-exponential}
\mbox{Cov}(G(t),G(s))&=&\int_{s}^{t}\int_{0}^{s}\phi(u-v)dvdu+K(s) \nonumber \\
&=& \frac{\alpha\beta(2\beta-\alpha)}{2(\beta-\alpha)^{4}} \left(e^{(\alpha -\beta)s} - e^{(\alpha -\beta)t} \right) \cdot \left( e^{(\beta-\alpha)s} -1\right) +K(s)\nonumber \\
&=& \frac{\beta^3 }{(\beta-\alpha)^3} s  + \frac{\alpha\beta(2\beta-\alpha)}{2(\beta-\alpha)^{4}} \left( -1  - e^{(\alpha -\beta)(t-s)} + e^{(\alpha -\beta)t} + e^{(\alpha -\beta)s} \right). \nonumber \\
\end{eqnarray*}
In this special case, we notice that $K(\cdot)$, the variance function of $G$, is nonlinear in $t$ in general. This is very different from the case when $N^{\mu}$ is a Poisson process (i.e., $h \equiv 0$) where $G$ becomes a standard Brownian motion. In addition, we find from \eqref{eq:Kt-exp} that when $h$ is a single exponential function, the variance function
$K(\cdot)$ have the following properties: $K(\cdot)$ is Lipschitz continuous, convex, and asymptotically linear as $t \rightarrow \infty$.


For a general exciting function $h,$ we next summarize important properties of $K(t)=\mbox{Var}(G(t))$ defined in \eqref{eq:var} and $\phi(t)$ defined in \eqref{eq:phi} in the following result. These properties provide us a better understanding about the variance of the limit Gaussian process $G$.  

\begin{proposition}\label{prop:Kt}
Under Assumption~1, the following hold:
\begin{itemize}
\item [(a)]
\begin{equation*}
\lim_{t\rightarrow\infty}\frac{K(t)}{t}=\frac{1}{(1-\Vert h\Vert_{L^{1}})^{3}}.
\end{equation*}


\item [(b)]
$\Vert\phi\Vert_{L^{1}}<\infty$, and the variance function $K(\cdot)$ is convex and Lipschitz continuous on $[0,\infty)$.

\item [(c)] If in addition $\int_{0}^{\infty}t^{2}h(t)dt<\infty$, then
\begin{align*}
&\lim_{t\rightarrow\infty}\left[K(t)-\frac{t}{(1-\Vert h\Vert_{L^{1}})^{3}}\right]
\\
&=\frac{1}{\pi(1-\Vert h\Vert_{L^{1}})^{3}}\int_{\mathbb{R}}\frac{1}{\omega^{2}}
\frac{(1-\Vert h\Vert_{L^{1}})^{2}-|1-\hat{h}(\omega)|^{2}}
{|1-\hat{h}(\omega)|^{2}}d\omega <0,
\nonumber
\end{align*}
where $\hat{h}$ is given by
\begin{eqnarray*}
\hat{h}(\omega) = \int_{0}^{\infty} e^{i\omega t} h(t) dt.
\end{eqnarray*}
\end{itemize}
\end{proposition}
The proof of this result is given in Appendix~\ref{sec:prop-Kt}.

Part (a) of this result is known in the literature, and we include it here mainly for completeness. The results in other parts appear to be new.

Having characterized the covariance and variance functions of $G$, we can now elaborate further properties of the Gaussian process $G$. We summarize them in the following result. The proof is given in Appendix~\ref{sec:G-prop}.

\begin{proposition}\label{prop:G-property}
Under Assumption~1,
the Gaussian process $G$ in Theorem~\ref{thm:FCLT} has stationary increments. In addition, the Gaussian process $G$ is not Markovian unless $h\equiv 0$. Furthermore,
the paths of $G$ are H\"{o}lder continuous of order $\gamma$ for every $\gamma<\frac{1}{2}$.
\end{proposition}

\begin{remark}
The increments of the Gaussian process $G$ are positively correlated and dependent in general. This is clear from \eqref{eq:cov-G} since for $s, \tau>0$
\begin{equation*}
\mbox{Cov}(G(s+\tau)- G(s),G(s))=\int_{s}^{s+\tau}\int_{0}^{s}\phi(u-v)dvdu,
\end{equation*}
which is nonzero and positive.
\end{remark}

\section{Infinite--server queues with self-exciting traffic} \label{sec:inf-server}
In this section we study infinite--server queues with high-volume self-exciting traffic, i.e., the arrival process is modeled by a univariate stationary Hawkes process. We establish limit theorems for such queues in Section~\ref{sec:limit}, characterize the limit process in Section~\ref{sec:property}, and discuss in detail the special case of exponential service time distributions in Section~\ref{sec:exp}.


\subsection{Limit theorems for $GI/\infty$ queues with self--exciting Hawkes traffic} \label{sec:limit}

In this section, we follow \cite{Kri1997} to establish the limit theorems for $GI/\infty$ queues with self--exciting Hawkes traffic.

We consider a sequence of infinite-server queueing models indexed by $\mu$ and
let $\mu \rightarrow \infty$. For each fixed $\mu$, the customers arrive to the $\mu-$th system according to a stationary univariate Hawkes process $N^{\mu}$ with a baseline intensity $\mu$ and an exciting function $h$. Hence, the average arrival rate is $\bar \lambda =\frac{\mu}{1-\Vert h\Vert_{L^{1}}}$. Write $Q^{\mu}(t)$ as the number of customers in the $\mu-$th system at time $t$.

We assume given an i.i.d. sequence of nonnegative random variables $\{\bar \eta_i: i \ge 1\}$ with a cumulative distribution function $F_0(x) = \mathbb{P}(\bar \eta_1 \le x)$ and another i.i.d. sequence of nonnegative random variables $\{\eta_i: i \ge 1\}$ with a cumulative distribution function $F(x) = \mathbb{P}(\eta_1 \le x)$. Assume $F_0(0)=F(0)=0$ for simplicity. The customers initially present in the infinite--server queueing system have remaining service times $\bar \eta_1, \ldots, \bar \eta_{Q^{\mu}(0)}$; the new arriving customers have service times $\eta_1, \eta_2, \ldots.$
All these service times, $Q^{\mu}(0)$, and the arrival process $N^{\mu}$ are assumed to be mutually independent. Then we have (see, e.g., \cite{Kri1997, Pang2010})
\begin{eqnarray*} \label{eq:Q-mu}
Q^{\mu}(t) = \sum_{i=1}^{Q^{\mu}(0)} 1_{\bar \eta_i >t} + \sum_{i=1}^{N^{\mu}(t)} 1_{\tau_i + \eta_i >t},
\end{eqnarray*}
where $\tau_i$ is the arrival time of the $i$-th new customer.

It follows from Theorem~3 in \cite{Kri1997} and our Theorem~\ref{thm:FCLT} that the following result holds. The proof is omitted.

\begin{proposition} \label{thm:IS1}
Suppose Assumption~1 holds.
Assume that for some constant $q_0$ and random variable $\xi,$
\begin{eqnarray} \label{eq:initial}
\sqrt{\mu} \left( \frac{Q^{\mu}(0)}{\mu} - q_0 \right) \Rightarrow \xi, \quad \text{as $\mu \rightarrow \infty.$}
\end{eqnarray}
Then the sequence of processes
$X^{\mu}$ defined by
\begin{eqnarray}
X^{\mu}(t) &=&{\sqrt{\mu}} \left( \frac{Q^{\mu}(t)}{\mu} - q_0 (1-F_0(t)) - \frac{1}{1-\Vert h\Vert_{L^{1}}} \cdot \int_{0}^t (1-F(t-u)) du \right), \label{eq:X-Q-1d}
\end{eqnarray}
as $\mu \rightarrow \infty,$
converges in distribution in $(D([0,\infty),\mathbb{R}),J_{1})$ to the process $X$ where
\begin{equation} \label{eq:X}
X(t) =(1- F_0(t)) \xi + \sqrt{q_0} \cdot W^0 (F_0(t))  +  \theta(t)  + \int_{0}^t  (1- F(t-u)) dG(u) .
\end{equation}
Here, $W^0=\{W^0(x): x \in [0,1]\}$ is a Brownian bridge,
$G$ is the mean-zero Gaussian process given in Theorem~\ref{thm:FCLT}, $\theta$ is a mean-zero Gaussian process with covariance function given by
\begin{equation} \label{eq:cov-theta}
\mathbb{E}[ \theta(s) \theta(t)] = \frac{1}{1-\Vert h\Vert_{L^{1}}} \cdot \int_{0}^s F(s-u) (1-F(t-u)) du, \quad \text{$0 \le s \le t$.}
\end{equation}
The random elements
$\xi, W^0, G, \theta$ are mutually independent.
\end{proposition}

\begin{remark}
The integral $ \int_{0}^t  (1- F(t-u)) dG(u) $ in \eqref{eq:X} is defined in a pathwise sense and is understood as the result of integration by parts.   See Theorem~3 in \cite{Kri1997}.
In addition, it is known in the literature  (see, e.g., \cite{Kri1997, Pang2010}) that one can represent the Gaussian process
$\theta$ as an integral with respect to a random field, that is,
\begin{equation*}
\theta(t) = - \int_{0}^t  \int_0^{t} 1_{s +x \le t} dU\left(\frac{s}{1-\Vert h\Vert_{L^{1}}}, F(x) \right),
\end{equation*}
where the
Kiefer process $U(\cdot, \cdot)$ is a two-parameter continuous centered Gaussian process on $\mathbb{R}_+ \times [0,1]$ with covariance function
\begin{equation*}
\mathbb{E}[ U(s, x) U(t,y) ]= (s \wedge t ) (x \wedge y - xy).
\end{equation*}
\end{remark}

As $\xi$ is independent of the other three Gaussian processes $W^0, G, \theta$, so for given $\xi = x_0 \in \mathbb{R},$ we obtain that the limit process $X$ in Proposition~\ref{thm:IS1} is Gaussian. We next discuss the properties of this Gaussian limit $X$ with a given initial condition $X(0)=\xi= x_0 \in \mathbb{R}.$

\subsection{Properties of the Gaussian process $X$ in \eqref{eq:X}} \label{sec:property}
In this section, we characterize the Gaussian limit process $X$ in Proposition~\ref{thm:IS1} by computing the mean, covariance function, and long-term behavior of $X$ with a given initial condition $X(0)= x_0 \in \mathbb{R}.$

It is clear from Proposition~\ref{thm:IS1} that for each fixed $t \ge 0,$ the mean of $X(t)$ is given by:
\begin{equation*}\label{eq:EXt}
\mathbb{E}[X(t)| X(0) = x_0]=(1-F_{0}(t)) x_0.
\end{equation*}

To compute the covariance of $X$, we can obtain from Proposition~\ref{thm:IS1} that for
$t\geq s \ge 0$,
\begin{align}
\text{Cov}(X(t),X(s))
&=q_{0}\text{Cov}(W^{0}(F_{0}(t)),W^{0}(F_{0}(s)))+\text{Cov}(\theta(t),\theta(s))
\nonumber \\
&\qquad
+\text{Cov}\left(\int_{0}^{t}(1-F(t-u))dG(u),\int_{0}^{s}(1-F(s-v))dG(v)\right).  \label{eq:covX}
\end{align}
By using the property of Brownian bridge, for $t\geq s$, we have
\begin{equation*}
q_{0}\text{Cov}(W^{0}(F_{0}(t)),W^{0}(F_{0}(s)))
=q_{0}F_{0}(s)(1-F_{0}(t)).
\end{equation*}
In addition, $\text{Cov}(\theta(t),\theta(s))$ is already given in \eqref{eq:cov-theta}. Hence,
it suffices to compute the last term in \eqref{eq:covX}.

We can directly compute that
\begin{align*}
&\text{Cov}\left(\int_{0}^{t}(1-F(t-u))dG(u),\int_{0}^{s}(1-F(s-v))dG(v)\right)
\nonumber \\
&=\mathbb{E}\left[\int_{0}^{t}(1-F(t-u))dG(u)\int_{0}^{s}(1-F(s-v))dG(v)\right]
\nonumber \\
&=\mathbb{E}\left[\int_{0}^{s}(1-F(t-u))dG(u)\int_{0}^{s}(1-F(s-v))dG(v)\right]
\nonumber \\
&\qquad\qquad\qquad
+\mathbb{E}\left[\int_{s}^{t}(1-F(t-u))dG(u)\int_{0}^{s}(1-F(s-v))dG(v)\right]
\nonumber
\\
&=\int_{0}^{s}(1-F(t-u))(1-F(s-u))dK(u)
\nonumber
\\
&\qquad\qquad\qquad\qquad
+\int_{0}^{s}\int_{s}^{t}(1-F(t-u))(1-F(s-v))\phi(v-u)dvdu
\nonumber
\\
&=\frac{1}{1-\Vert h\Vert_{L^{1}}}\int_{0}^{s}(1-F(t-u))(1-F(s-u))du
\nonumber
\\
&\qquad\qquad\qquad\qquad
+\int_{0}^{s}\int_{0}^{t}(1-F(t-u))(1-F(s-v))\phi(v-u)dudv,
\end{align*}
where we used \eqref{eq:var}.
Thus, we obtain the following result.

\begin{proposition}[Covariance function of the Gaussian process $X$ in \eqref{eq:X}] \label{prop:covX}
Given $X(0)= x_0 \in \mathbb{R}.$
For $0 \le s \le t,$ we have
\begin{align*}
\text{Cov}(X(s),X(t))
&=q_{0}F_{0}(s)(1-F_{0}(t)) + \frac{1}{1-\Vert h\Vert_{L^{1}}}\int_{0}^{s}(1-F(t-u))du
\\
&\qquad\qquad
+\int_{0}^{s}\int_{0}^{t}(1-F(t-u))(1-F(s-v))\phi(v-u)dudv,
\end{align*}
where $\phi$ is determined by the exciting function $h$ from Equation~\eqref{eq:phi}, and $\phi(x) = \phi(-x)$ for $x<0$.
\end{proposition}
An immediate observation from this result is that the covariance density of the traffic input Hawkes process, together with the service time distributions, leads to a delicate correlation structure of the limiting scaled queue length process $X$.

From Proposition~\ref{prop:covX}, we can immediately find that given $X(0)= x_0 \in \mathbb{R},$
\begin{align}
\text{Var}(X(t))
&=q_{0}F_{0}(t)(1-F_{0}(t))
+\frac{1}{1-\Vert h\Vert_{L^{1}}}\int_{0}^{t}(1-F(u))du
\nonumber \\
&\qquad\qquad
+\int_{0}^{t}\int_{0}^{t}(1-F(u))(1-F(v))\phi(v-u)dudv. \label{eq:varXt}
\end{align}

Note that $(1-F_0(t)) \xi$ converges to 0 almost surely as $t\rightarrow\infty$.
In view of \eqref{eq:X} and by letting $t\rightarrow\infty$ in \eqref{eq:varXt}, we get the following result about the long-term behavior of the limiting process $X$.


\begin{corollary}\label{prop:X-infty}
As $t \rightarrow \infty,$ the sequence of random variables $X(t)$ in \eqref{eq:X} converges in distribution to $X(\infty)$ which is a Gaussian random variable with mean zero and variance
\begin{align*}
\text{Var}(X(\infty))
=\frac{1}{1-\Vert h\Vert_{L^{1}}}\int_{0}^{\infty}(1-F(u))du
+\int_{0}^{\infty}\int_{0}^{\infty}(1-F(u))(1-F(v))\phi(v-u)dudv.
\end{align*}
\end{corollary}

\subsection{A special case: exponential service times} \label{sec:exp}
In this section, we discuss in detail the special case that service times of each customer are mutually independent and exponentially distributed. Without loss of generality, we consider service time distribution with mean one.


 Then we have the following result.

\begin{proposition} \label{thm:IS}
Suppose Assumption~1 holds.
Assume \eqref{eq:initial} and
\begin{equation}\label{ass1}
q_0= \frac{1}{1-\Vert h\Vert_{L^{1}}}, \quad \text{and} \quad F(x)=F_0(x) =1- e^{-x}, \quad \text{$x \ge 0$.}
\end{equation}
Then as $\mu \rightarrow \infty,$
the sequence of processes
$X^{\mu}$ in \eqref{eq:X-Q-1d} converges in distribution to the process $X_e$ with continuous sample paths in $(D([0,\infty),\mathbb{R}),J_{1})$ and
\begin{equation*}\label{eq:Xe1}
X_e(t) =\xi \cdot e^{-t}  + \frac{1}{\sqrt{1-\Vert h\Vert_{L^{1}}}}\cdot e^{-t} \cdot \int_{0}^t e^s dB(s) + e^{-t} \cdot \int_{0}^t e^s dG(s),
\end{equation*}
or equivalently,
\begin{equation} \label{eq:Xe}
X_e(t) =\xi - \int_{0}^t X_e(s) ds  +  \frac{1}{\sqrt{1-\Vert h\Vert_{L^{1}}}} \cdot B(t) +G(t),
\end{equation}
where $G$ is the mean-zero Gaussian process given in Theorem~\ref{thm:FCLT}, $B$ is a standard Brownian motion, and $\xi, G, B$ are mutually independent. In addition, the Gaussian process $X_e$ is non--Markovian unless $h \equiv 0$.
\end{proposition}

The proof of the weak convergence in this result immediately follows from Proposition~\ref{thm:IS1} and Part II of Theorem~3 in \cite{Kri1997}. The non-Markovian property of $X_e$ is also evident given the non-Markovian property of $G$ in Proposition~\ref{prop:G-property}. We omit the proof.

Note under the assumptions in \eqref{ass1}, one can readily verify from \eqref{eq:X-Q-1d} that
\begin{eqnarray} \label{eq:X-mu-t-exp}
X^{\mu}(t) ={\sqrt{\mu}} \left( \frac{Q^{\mu}(t)}{\mu} - q_0  \right) = \frac{1}{\sqrt{\mu}} \left(Q^{\mu}(t) - \frac{\mu}{1-\Vert h\Vert_{L^{1}}}\right) = \frac{1}{\sqrt{\mu}} \left(Q^{\mu}(t) - \bar \lambda \right).
\end{eqnarray}
In the classical case where the traffic is Poisson, i.e., $h \equiv 0$, it is well known that $G$ reduces to a standard Brownian motion, and the sequence $X^{\mu}$ converges in distribution to the limit process $X_e$ where $X_e$ is an Ornstein--Uhlenbeck (OU) diffusion process (driven by a Brownian motion) which is Markovian.
When the traffic model is a Hawkes process and the exciting function $h$ is nonzero, Equation~\eqref{eq:Xe} suggests that the limit process $X_e$ can be viewed as an Ornstein-Uhlenbeck (OU) process driven by the centered Gaussian process $Y$ where
\begin{equation*}
Y(t) := \frac{1}{\sqrt{1-\Vert h\Vert_{L^{1}}}} \cdot B(t) +G(t), \quad \text{for $t \ge 0.$}
\end{equation*}
We explore additional properties of the process $X_e$ in the next section.

\subsubsection{Properties of the Gaussian-driven OU process $X_e$}
From the results in Section~\ref{sec:property}, we can immediately obtain the mean, the covariance function and the long-term behavior of the Gaussian-driven OU process $X_e$.

\begin{proposition}\label{prop:Xe-cov}
Assume that $X_e(0)=x_0 \in \mathbb{R}$. Then we have
\begin{equation*}
\mathbb{E}[X_e(t)|X_e(0)=x_0]= x_0 \cdot e^{-t}, \quad \text{$t \ge 0$.}
\end{equation*}
In addition, for $0\leq s\leq t$, we have
\begin{equation*}  \label{eq:cov_Xe}
\mbox{Cov}(X_e(s),X_e(t))
=e^{-s-t}\bigg[ (e^{2s}-1) \cdot \frac{1}{1 - \Vert h \Vert_{L^1}}
+\int_{0}^{t}\int_{0}^{s} e^{u+ v} \phi(u-v)dv du
\bigg],
\end{equation*}
where $\phi$ is determined by the exciting function $h$ from Equation~\eqref{eq:phi}.
Finally, as $t \rightarrow \infty,$ the sequence of random variables $X_e(t)$ in \eqref{eq:Xe} converges in distribution to $X_e(\infty)$ which is a Gaussian random variable with mean zero and variance
\begin{equation} \label{eq:var-Xe-infty}
\mbox{Var}(X_e(\infty))
=\int_{0}^{\infty}e^{-t}\phi(t)dt+ \frac{1}{1-\Vert h\Vert_{L^{1}}}.  
\end{equation}
\end{proposition}
One can obtain an explicit formula for the covariance function and long-term limit of $X_e$ in the special case $h(t)=\alpha e^{-\beta t}$ where $0 \le \alpha< \beta$.
Recall when $h(t)=\alpha e^{-\beta t}$, we have $\phi$ given in \eqref{eq:phi-exp}.
Therefore, from Proposition~\ref{prop:Xe-cov} we can compute that for $0\leq s\leq t$:
\begin{align*}
&\mbox{Cov}(X_e(s),X_e(t))   \\
&=e^{-s-t}\bigg[(e^{2s}-1)\frac{\beta}{ \beta - \alpha}
+\frac{\alpha\beta(2\beta-\alpha)}{2(\beta-\alpha)^{2}}
\int_{0}^{t}\int_{0}^{s}e^{(u+v)}e^{-(\beta-\alpha)|u-v|}dvdu\bigg]
\\
&=e^{-s-t}\bigg[(e^{2s}-1)\frac{\beta}{ \beta - \alpha}
+\frac{\alpha\beta(2\beta-\alpha)}{2(\beta-\alpha)^{2}}
\frac{(e^{(1-\beta+\alpha)t}-e^{(1-\beta+\alpha)s})
(e^{(\beta-\alpha+1)s}-1)}{(1+\beta-\alpha)(1-\beta+\alpha)}
\\
&\qquad
+\frac{\alpha\beta(2\beta-\alpha)}{2(\beta-\alpha)^{2}}
\frac{1}{1+\beta-\alpha}\left(\frac{e^{2s}-1}{2}-\frac{e^{(1-\beta+\alpha)s}-1}{1-\beta+\alpha}\right)
\\
&\qquad\qquad
+\frac{\alpha\beta(2\beta-\alpha)}{2(\beta-\alpha)^{2}}
\frac{1}{1-\beta+\alpha}
\left(\frac{e^{2s}-e^{(1-\beta+\alpha)s}}{1+\beta-\alpha}
-\frac{e^{2s}-1}{2}\right)\bigg].
\end{align*}
In addition,
\begin{equation} \label{VAR1}
\mbox{Var}(X_e(\infty))
=\frac{\alpha\beta(2\beta-\alpha)}{2(\beta-\alpha)^{2}}\cdot \frac{1}{1+\beta-\alpha}+\frac{\beta}{ \beta - \alpha}.
\end{equation}

When $h$ is not a single exponential function, let us discuss how to compute the Laplace
transform of $\phi$ in \eqref{eq:var-Xe-infty} in general. It is proved in Lemma 11 of \cite{ZhuI}
that if $h$ is positive, continuous and integrable, $h$ can be approximated
by a sum of exponentials in both $L^{1}$ and $L^{\infty}$ norms,
that is there exist $\alpha_{i}\in\mathbb{R}$ and $\beta_{i}\in\mathbb{R}_{+}$ so that
$h_{n}(t):=\sum_{i=1}^{n}\alpha_{i}e^{-\beta_{i}t}\geq 0$ for every $t>0$
and $h_{n}\rightarrow h$ in both $L^{1}$ and $L^{\infty}$ norms.
For such a general exciting function $h$, let us write $h(t)=\sum_{i=1}^{\infty}\alpha_{i}e^{-\beta_{i}t}$.
Thus, for any $\omega>0$,
\begin{align*}
\int_{0}^{\infty}e^{-\omega t}\int_{0}^{\infty}h(t+v)\phi(v)dvdt
&=\int_{0}^{\infty}e^{-\omega t}\int_{0}^{\infty}\sum_{i=1}^{\infty}\alpha_{i}e^{-\beta_{i}(t+v)}\phi(v)dvdt
\\
&=\sum_{i=1}^{\infty}\frac{\alpha_{i}}{\beta_{i}+\omega}\tilde{\phi}(\beta_{i}),
\end{align*}
where $\tilde{g}(\omega):=\int_{0}^{\infty}e^{-\omega t}g(t)dt$ for $\omega>0$ and a given function $g$.
Thus, by taking Laplace transform on both sides of \eqref{eq:phi}, we get
\begin{equation}\label{phiLaplace}
\tilde {\phi}(\omega)=\frac{\tilde{h}(\omega)}{(1-\tilde{h}(\omega))(1-\Vert h\Vert_{L^{1}})}
+\frac{1}{1-\tilde{h}(\omega)}\sum_{i=1}^{\infty}\frac{\alpha_{i}}{\beta_{i}+\omega}\tilde{\phi}(\beta_{i}).
\end{equation}
By letting $\omega=\beta_{i}$, $i=1,2,\ldots$, we get
\begin{equation*}
\tilde{\phi}(\beta_{i})=\frac{\tilde {h}(\beta_{i})}{(1-\tilde{h}(\beta_{i}))(1-\Vert h\Vert_{L^{1}})}
+\frac{1}{1-\tilde{h}(\beta_{i})}\sum_{j=1}^{\infty}\frac{\alpha_{j}}{\beta_{j}+\beta_{i}}\tilde{\phi}(\beta_{j}).
\end{equation*}
Let $\mathbf{\tilde X}$ denote the vector
with $\mathbf{\tilde X}_{i}=\tilde{\phi}(\beta_{i})$, and $\mathbf{R}$ be the vector
with $\mathbf{R}_{i}=\frac{\tilde{h}(\beta_{i})}{(1-\tilde {h}(\beta_{i}))(1-\Vert h\Vert_{L^{1}})}$
and $\mathbf{M}$ be the matrix with entries $\mathbf{M}_{ij}=\frac{1}{1-\tilde{h}(\beta_{i})}\frac{\alpha_{j}}{\beta_{j}+\beta_{i}}$,
and finally $\mathbf{I}$ be the identity matrix. Thus, we have
\begin{equation} \label{eq:tildeX}
\mathbf{\tilde X}=\mathbf{R}+\mathbf{M}\mathbf{\tilde X},
\end{equation}
which implies that $\mathbf{\tilde X}=\mathbf{R}(\mathbf{I}-\mathbf{M})^{-1}$ provided that
$\mathbf{I}-\mathbf{M}$ is invertible, which holds
if for example the spectral radius of $\mathbf{M}$ is strictly less than $1$.
In practice, if we consider $h(t)=\sum_{i=1}^{d}\alpha_{i}e^{-\beta_{i}t}$, for some finite $d\in\mathbb{N}$,
where $\beta_{i}>0$, $\alpha_{i}\in\mathbb{R}$
and $h(t)\geq 0$ for every $t\geq 0$, then one can readily obtain $\mathbf{R}$ and $\mathbf{M}$, and hence
$\mathbf{\tilde X}$ and $\tilde{\phi}(\beta_{i})$ can be easily solved.
Once the values of $\tilde{\phi}(\beta_{i})$ are determined, so is the Laplace transform of $\phi$
given in \eqref{phiLaplace}. An example to illustrate this procedure will be provided in the next section (Example 2).

\subsubsection{Gaussian approximations and numerical experiments} \label{sec:Gaussian-approx}

Note that Proposition~\ref{thm:IS} and \eqref{eq:X-mu-t-exp} suggest that when the baseline intensity $\mu$ of the Hawkes arrival process is large, we can heuristically approximate the steady-state distribution of the number of customers $Q^{\mu}(\infty)$ in the $Hawkes/M/\infty$ queue as follows:
\begin{equation} \label{eq:approx-Gaussian}
Q^{\mu}(\infty) \approx \bar \lambda + \sqrt{\mu} X_e(\infty),
\end{equation}
where the random variable $\bar \lambda + \sqrt{\mu} X_e(\infty)$ follows a normal distribution with mean zero, and variance $\mu \cdot Var(X_e(\infty))$ where $Var(X_e(\infty))$ is given in \eqref{eq:var-Xe-infty}. A more precise statement of the approximation in \eqref{eq:approx-Gaussian} is
\begin{equation} \label{eq:approx-gau}
\mathbb{P}(Q^{\mu}(\infty) = i ) \approx  \frac{1}{\sigma} f\left( \frac{i - \bar \lambda }{\sigma}\right) , \quad \text{for $i=0, 1, \ldots,$}
\end{equation}
where $\sigma:= \sqrt{\mu \cdot \text{Var}(X_e(\infty))} $, and $f$ is the probability density function of a standard normal distribution.

We now present numerical experiments to demonstrate that the Gaussian approximation in \eqref{eq:approx-gau} is effective by making comparisons with simulations of the $Hawkes/M/\infty$ queue.  We consider two examples.

\textit{Example 1.} We first consider the Hawkes input process with a single exponential function:
\begin{equation*}\label{eq:h1}
h_1(t) = \frac{1}{2} e^{-t}.
\end{equation*}
 It is clear that $\Vert h_1\Vert_{L^{1}} = \frac{1}{2}$. Suppose the stationary Hawkes input process has a baseline intensity $\mu$. Then we can infer from \eqref{VAR1} that the Gaussian random variable $\bar \lambda + \sqrt{\mu} X_e(\infty)$ has mean $\bar \lambda = 2 \mu$ and variance $\sigma^2 = \mu \cdot\text{Var}(X_e(\infty)) = 3 \mu$, where we have used \eqref{VAR1} to find that
 $\text{Var}(X_e(\infty))=3.$

We now compare the Gaussian approximation in \eqref{eq:approx-gau} with simulations in Figure~\ref{fig1}. We observe that the approximation
agrees with the simulation results well, even for moderately large $\mu$ such as twenty.

\begin{figure}[htb]
	\centering
	\subfigure[$\mu =20$]{
		\begin{minipage}[b]{0.44\textwidth}
			\centering
			\includegraphics[width=\textwidth]{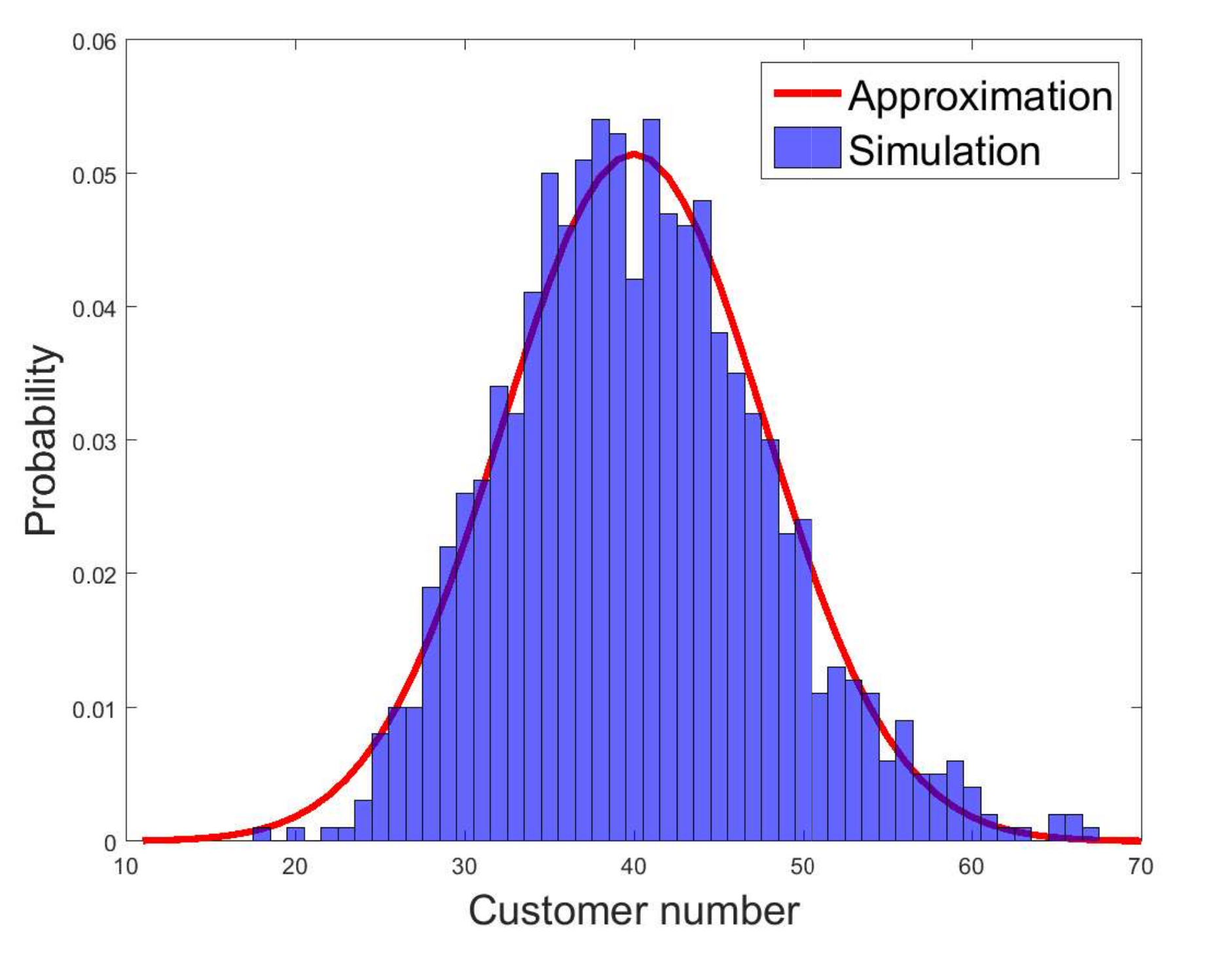}
	\end{minipage}	}			
	\subfigure[$\mu = 100$]{
		\label{fig:mini:subfig:1}
		\begin{minipage}[b]{0.45\textwidth}
			\centering
			\includegraphics[width=\textwidth]{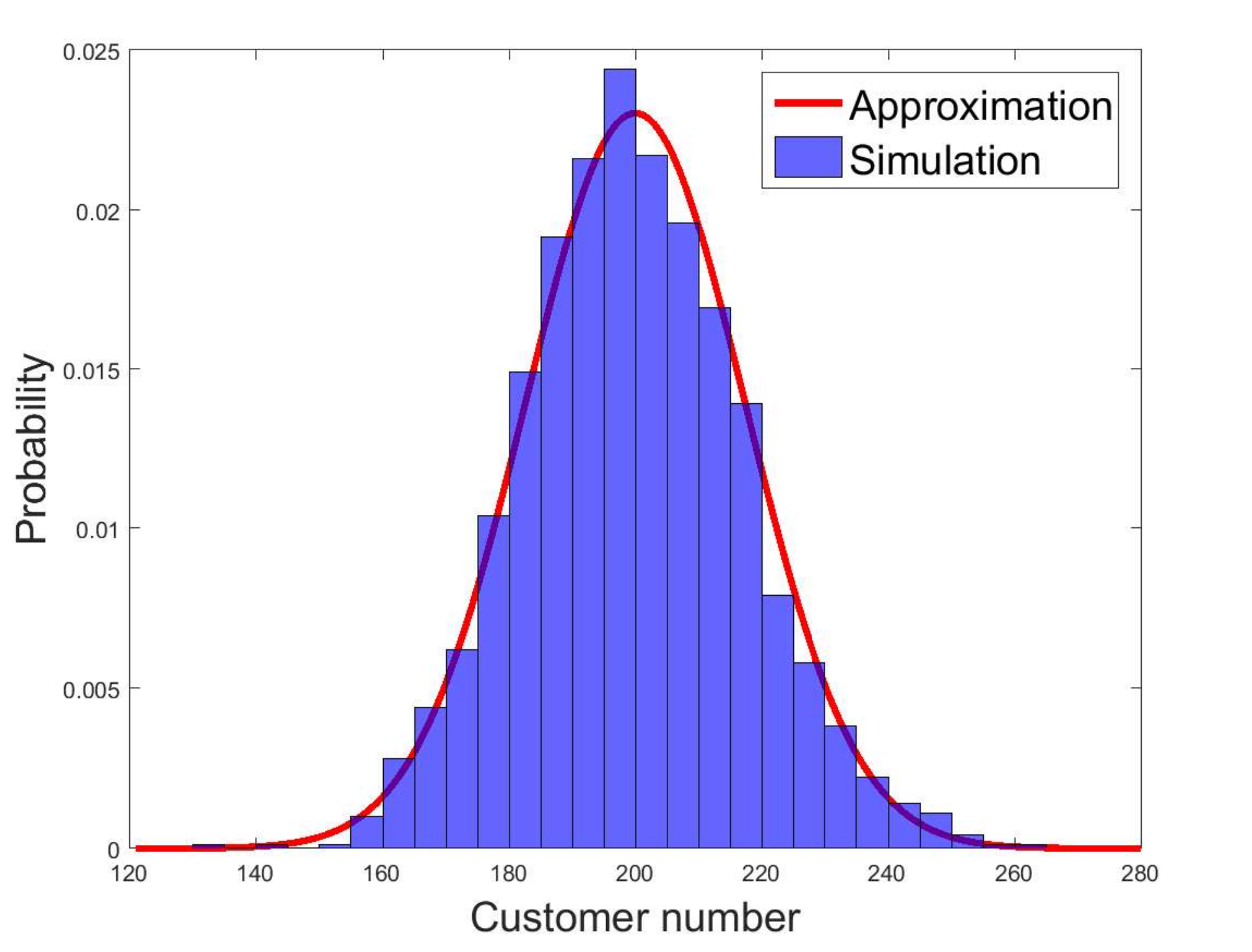}
	\end{minipage}}	
	\caption{The steady-state distribution of the number of customers $Q^{\mu}(\infty)$ in the $Hawkes/M/\infty$ queue where the input is a stationary Hawkes process with a baseline intensity $\mu$ and an exciting function $h_1(t) = \frac{1}{2} e^{-t}$. The service time distribution is exponential with mean one. The Gaussian approximation in \eqref{eq:approx-gau} is compared with simulations.}
	\label{fig1}
\end{figure}

\textit{Example 2:}
We next consider a Hawkes input process with an exciting function which is a sum of exponentials:
\begin{equation*} \label{eq:h2}
h_2(t) = \frac{1}{10} e^{-\frac{1}{4}t} + \frac{2}{5} e^{-4t}.
\end{equation*} It is also clear that $\Vert h_2\Vert_{L^{1}} = 0.5$. In this case, when the baseline intensity of the Hawkes process is $\mu$, we have the Gaussian random variable $\bar \lambda + \sqrt{\mu} X_e(\infty)$ has mean $\bar \lambda = 2 \mu$ and variance $\sigma^2 = \mu \cdot Var(X_e(\infty)) = 2.5246 \mu$.

 To see this, we compute $Var(X_e(\infty))$ using \eqref{eq:var-Xe-infty} and Equations~\eqref{phiLaplace} and \eqref{eq:tildeX}.
Note that $\tilde h_2 (1) =\int_0^{\infty} e^{-t}h_2(t) dt = 0.16$, and hence
\begin{equation} \label{eq:t-phi1}
\int_0^{\infty} e^{-t} \phi(t) dt = \tilde \phi(1) = \frac{0.16}{0.84} \cdot 2 + \frac{0.08}{0.84} \cdot    (\tilde \phi(0.25) + \tilde \phi(4)).
\end{equation}
One can readily verify from the expression of $h_2$ that
\begin{equation*}
\mathbf{R} = [0.8333, 0.1587] , \quad \text{and} \quad \mathbf{M} = \begin{bmatrix}
    0.2833       & 0.1333  \\
    0.0254     & 0.054  \\
\end{bmatrix}.
\end{equation*}
This yields
\begin{equation*}
\mathbf{\tilde X} = [\tilde \phi(0.25), \tilde \phi(4)]=\mathbf{R}(\mathbf{I}-\mathbf{M})^{-1}= [1.1745, 0.3333].
\end{equation*}
On combining \eqref{eq:t-phi1} and \eqref{eq:var-Xe-infty} we deduce that when the traffic is a Hawkes process with an exciting function $h_2$, then we have
\begin{equation*}
\mbox{Var}(X_e(\infty))
= \tilde \phi(1) + 2 = 2.5246.
\end{equation*}

\begin{figure}[htb]
	\centering
	\subfigure[$\mu =20$]{
		\begin{minipage}[b]{0.45\textwidth}
			\centering
			\includegraphics[width=\textwidth]{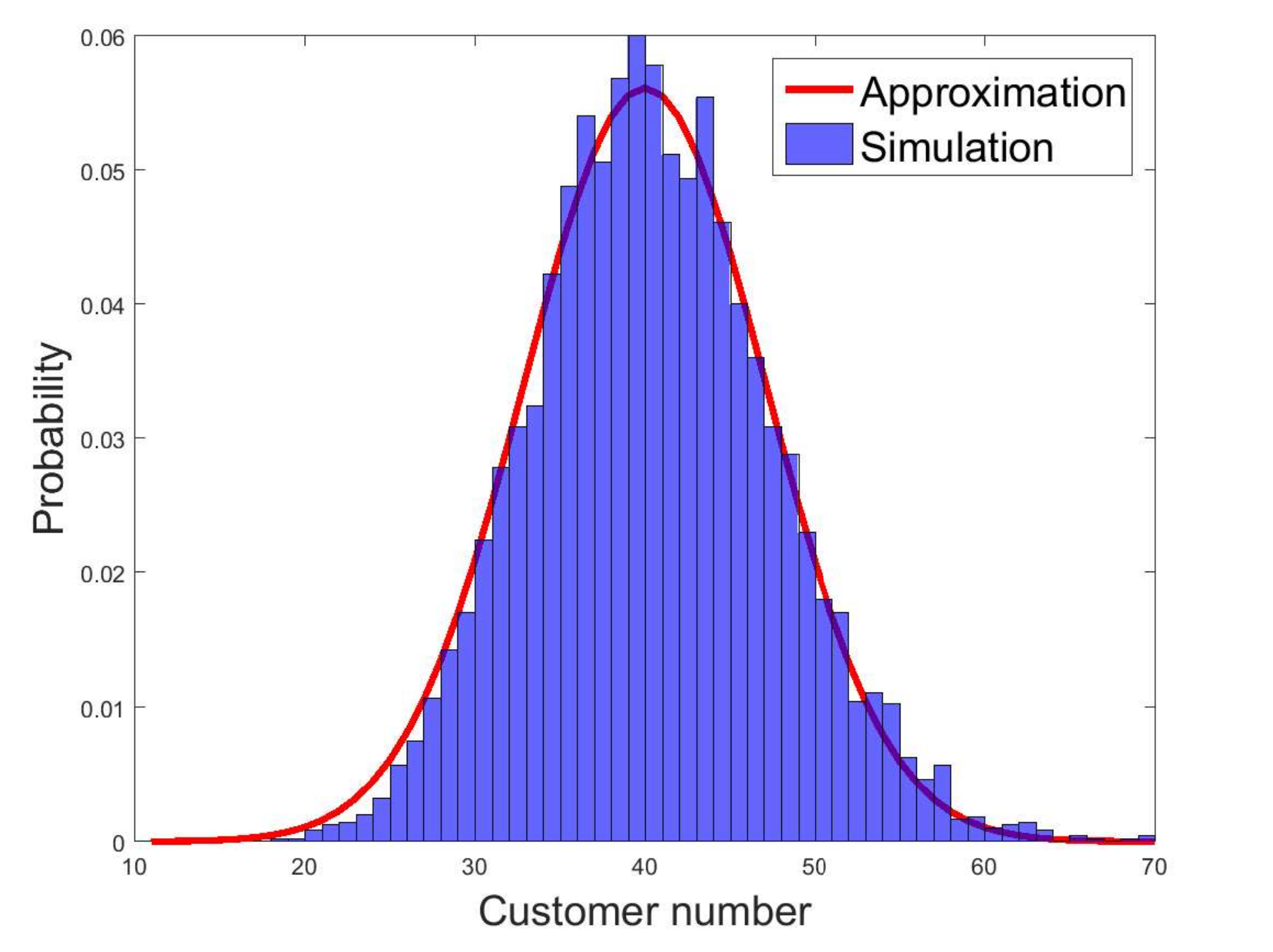}
	\end{minipage}	}			
	\subfigure[$\mu = 100$]{
		\label{fig:mini:subfig:1}
		\begin{minipage}[b]{0.44\textwidth}
			\centering
			\includegraphics[width=\textwidth]{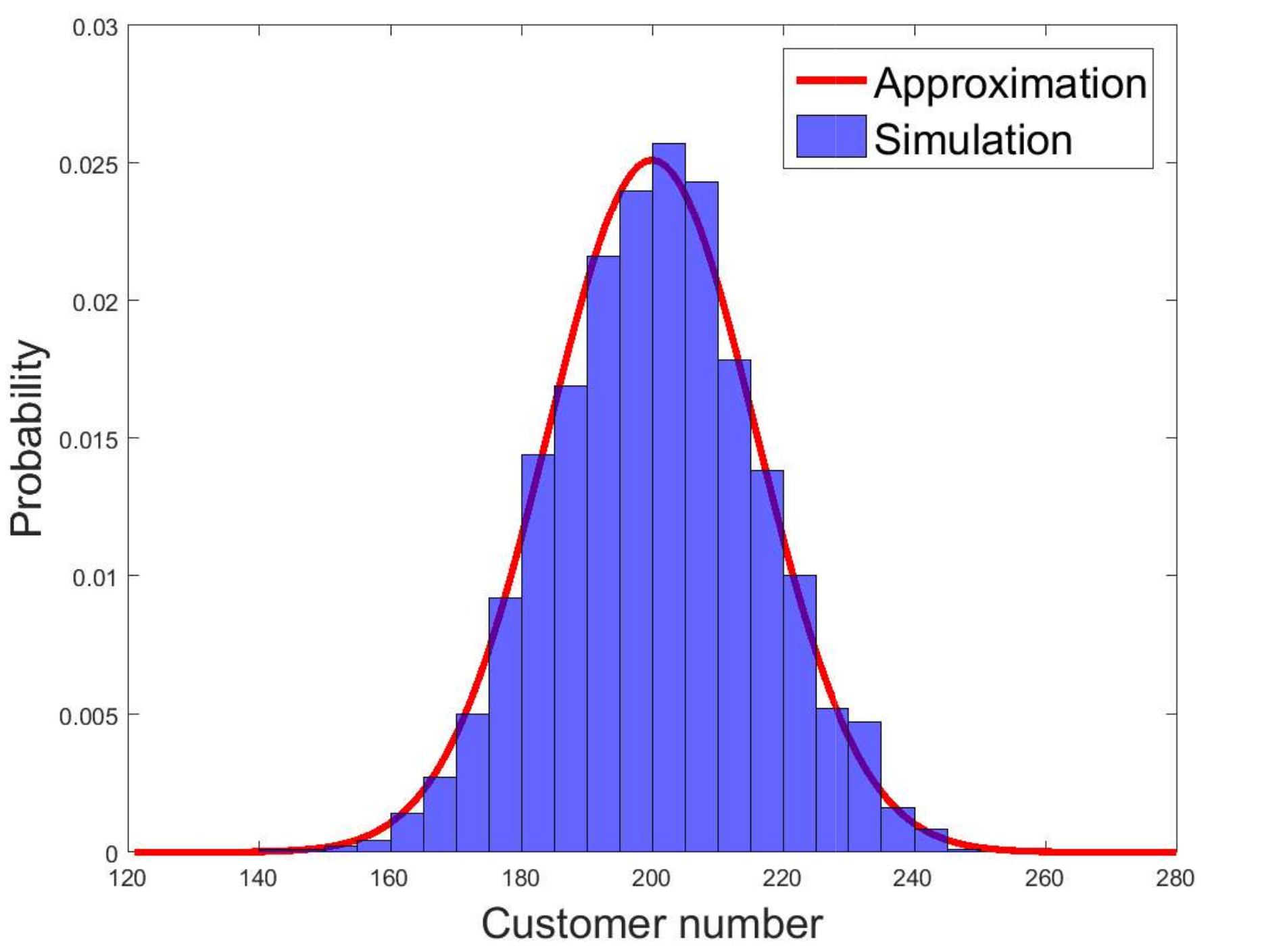}
	\end{minipage}}	
	\caption{The steady-state distribution of the number of customers $Q^{\mu}(\infty)$ in the $Hawkes/M/\infty$ queue where the input is a stationary Hawkes process with a baseline intensity $\mu$ and an exciting function $h_2(t) =\frac{1}{10} e^{-\frac{t}{4}} + \frac{2}{5} e^{-4t}$. The service time distribution is exponential with mean one. The Gaussian approximation in \eqref{eq:approx-gau} is compared with simulations.}
	\label{fig2}
\end{figure}

We next demonstrate in Figure~\ref{fig2} that the Gaussian approximation in \eqref{eq:approx-gau} is effective by comparing with simulations of the infinite--server queue with Hawkes input where the exciting function is $h_2$. We also observe that the Gaussian approximation for the steady--state customer number $Q^{\mu}(\infty)$  agrees with the simulation results very well.

\section{Infinite-server queues with mutually--exciting traffic} \label{sec:multi-hawkes}
In this section we extend Theorem~\ref{thm:FCLT} to multivariate stationary Hawkes processes, and establish limit theorems for infinite--server queues with multivariate Hawkes traffic.

\subsection{FCLT for multivariate stationary Hawkes processes}
In this section, we establish an FCLT for multivariate stationary Hawkes processes.

We consider a $k$-dimensional stationary Hawkes process $\mathbb{N}^{(\mu)}=(N^{\mu,1},N^{\mu,2},\ldots,N^{\mu,k})$,
where $N^{\mu,i}$ has the intensity:
\begin{equation} \label{eq:intensity-multi}
\lambda^{\mu,i}(t)=\mu p_{i}+\sum_{j=1}^{k}\int_{-\infty}^{t-}h_{ij}(t-s)N^{\mu,j}(ds),
\end{equation}
where $p_{i} \ge 0$ for $1\leq i\leq k$, and the exciting kernel $h_{ij}(\cdot)$ satisfies Assumption~\ref{assump1}. With slight abuse of notations, we still use $\mu>0$ as a scaling parameter, and study the limit of $\mathbb{N}^{\mu}$ as we send $\mu \rightarrow \infty$. Note that for each fixed $\mu>0$, we can obtain from \eqref{BarLambdaDefn} that for each $t,$
\begin{equation*} \label{eq:lambda-a}
\mathbb{E}[\mathbb{N}^{(\mu)}(t)]= \bar{\lambda} t =\mu a t,
\end{equation*}
where $\bar \lambda=(\bar \lambda_i)_{1 \le i \le k}$ is the average arrival rate and the vector $a=(a_i)_{1 \le i \le k}$ is given by
\begin{equation} \label{eq:a}
a := \left( \mathbb{I} - \mathbb{H}\right)^{-1} \cdot p,
\end{equation}
with $p=(p_i)_{1 \le i \le k}$, and $\mathbb{H}= (\Vert h_{ij}\Vert_{L^{1}})_{1\leq i,j\leq k}$.

Similar as in the univariate case, we can infer from the immigration-birth representation of multivariate Hawkes processes, when $\mu$ is a positive integer,
that the multivariate stationary Hawkes process $\mathbb{N}^{(\mu)}$ can be written as the sum of i.i.d. copies of $\mathbb{N}^{(1)}$, see e.g. \cite{JHR}. The covariance density of $\mathbb{N}^{(1)}$, which we still use the notation $\Phi=(\Phi_{ij})_{1\leq i,j\leq k}$ as in \eqref{eq:cov-density}, is given by
\begin{equation}\label{cov-multi}
\Phi(\tau)=h(\tau)\text{diag}(a)+\int_{-\infty}^{\tau}h(\tau-v)\Phi(v)dv,  \quad \text{for $\tau\geq 0$,}
\end{equation}
and $\Phi_{ij}(-\tau)=\Phi_{ji}(\tau)$ for every $\tau>0$ and $1\leq i,j\leq k$.
Here $\text{diag}(a)$ is the diagonal matrix with entries $a_{i}$'s on the diagonal,
and $h(t)=(h_{ij}(t))_{1\leq i,j\leq k}$. The variance function of $\mathbb{N}^{(1)}$, which we still use the notation
$\mathbb{K}(t)=(K_{ij}(t))_{1\leq i,j\leq k}$ as in \eqref{eq:var-function}, is given by
\begin{equation}\label{var-multi}
\mathbb{K}(t):=\text{diag}(a)t+2\int_{0}^{t}\int_{0}^{t_{2}}\Phi(t_{2}-t_{1})dt_{1}dt_{2}.
\end{equation}

Then we can obtain the following limit theorem for multivariate stationary Hawkes processes, which extends Theorem~\ref{thm:FCLT}. The proof is given in Appendix~\ref{sec:proof-multivariate}.



\begin{theorem}\label{thm:multi-FCLT}
Under Assumption~\ref{assump1}, we have as $\mu\rightarrow\infty$,
\begin{equation*}
\frac{\mathbb{N}^{(\mu)}(t)-\bar \lambda t}{\sqrt{\mu}}\Rightarrow\mathbb{G}(t),
\end{equation*}
in $(D([0,\infty),\mathbb{R}^{k}),J_{1})$, where $\mathbb{G}= (\mathbb{G}_i)_{1 \le i \le k}$ is a mean-zero almost surely continuous $k-$dimensional Gaussian process
with the covariance function
\begin{equation*}\label{eq:multi-cov-G}
\mbox{Cov}(\mathbb{G}(t),\mathbb{G}(s))=\int_{s}^{t}\int_{0}^{s}\Phi(u-v)dvdu+\mathbb{K}(s), \quad \text{for $t \ge s$,}
\end{equation*}
where $\Phi$ is given in \eqref{cov-multi} and $\mathbb{K}$ is given in \eqref{var-multi}.
\end{theorem}



\subsection{Limit theorem for $GI/\infty$ queues with multivariate Hawkes traffic}
In this section, we rely on Theorem~\ref{thm:multi-FCLT} to develop approximations for infinite--server queues with high-volume multivariate stationary Hawkes traffic.
Such a queueing model can be viewed as a multi-class queueing model with correlated arrivals as mutually--exciting Hawkes processes.

We first establish limit theorems for such queues.
Similar as in Section~\ref{sec:limit}, we consider a sequence of infinite-server queueing models indexed by $\mu \in \mathbb{R}_{+}$ and
let $\mu \rightarrow \infty$. For each fixed $\mu$, there are $k$ classes of customers arriving to the $\mu-$th system according to a stationary $k-$dimensional Hawkes process $\mathbb{N}^{(\mu)}$ with a baseline intensity vector $\mu \cdot {p}$ and an exciting kernel $(h_{ij})_{1\leq i,j\leq k}$. That is, the arrival process of customer class $i$ is the $i-$th component of the multivariate stationary Hawkes process $\mathbb{N}^{(\mu)}$.

In addition, for $1 \le i \le k$, each customer class $i$ may have different service requirements. For each $i$,
we assume given an i.i.d. sequence of nonnegative random variables $\{\bar \eta_{i,j}: j \ge 1\}$ with a cumulative distribution function $F_{i0}(x) = \mathbb{P}(\bar \eta_{i,j} \le x)$ and another i.i.d. sequence of nonnegative random variables $\{\eta_{i,j}: j \ge 1\}$ with a cumulative distribution function $F_i(x) = \mathbb{P}(\eta_{i,1} \le x)$. Assume $F_{i0}(0)=F_i(0)=0$ for all $i$ for simplicity. The customers of class $i$ initially present in the infinite--server queueing system have remaining service times $\bar \eta_{i,1}, \bar \eta_{i,2}, \ldots$; the new arriving customers of class $i$ have service times $\eta_{i,1}, \eta_{i,2}, \ldots.$
All these service times, the random initial numbers of customers of each class denoted by $\mathbb{Q}_1^{\mu}(0), \ldots, \mathbb{Q}_k^{\mu}(0)$, and the multivariate Hawkes arrival process $\mathbb{N}^{(\mu)}$ are assumed to be mutually independent.

Denote $\mathbb{Q}_i^{\mu}(t) $ as the number of customers of class $i$ in the $\mu-$th system at time $t$, and write the vector $\mathbb{Q}^{\mu}(t) := (\mathbb{Q}_i^{\mu}(t))_{1 \le i \le k}$.
Given Theorem~\ref{thm:multi-FCLT}, we can then obtain the following result. Recall the vector $a=(a_i)_{1 \le i \le k}$ given in \eqref{eq:a}.

\begin{proposition} \label{thm:IS-multi-gen}
Suppose Assumption~1 holds.
Assume that for some vector constant $q=(q_{10}, \ldots, q_{k0})$ and random vector $(\xi_1, \ldots, \xi_k),$
\begin{eqnarray} \label{eq:initial}
\sqrt{\mu} \left( \frac{\mathbb{Q}^{\mu}(0)}{\mu} - q \right) \Rightarrow (\xi_1, \ldots, \xi_k), \quad \text{as $\mu \rightarrow \infty.$}
\end{eqnarray}
Then the sequence of $k-$dimensional processes
$\mathbb{X}^{\mu}$ with its $i-$th component defined by
\begin{eqnarray}
\mathbb{X}_i^{\mu}(t) &=&{\sqrt{\mu}} \left( \frac{\mathbb{Q}_i^{\mu}(t)}{\mu} - q_{i0} (1-F_{i0}(t)) - a_i \cdot \int_{0}^t (1-F_i(t-u)) du \right), \label{eq:X-Q}
\end{eqnarray}
as $\mu \rightarrow \infty,$
converges in distribution in $(D([0,\infty),\mathbb{R}^k),J_{1})$ to the process $\mathbb{X}=(\mathbb{X}_i)_{1 \le i \le k}$ where
\begin{eqnarray*} \label{eq:X-multi-gen}
\mathbb{X}_i(t) &=&(1- F_{i0}(t)) \xi_i + \sqrt{q_{i0}} \cdot W^{i0} (F_{i0}(t))    + \int_{0}^t  (1- F_i(t-u)) d\mathbb{G}_i(u) \nonumber \\
&& - \int_0^t \int_0^t  1_{s +x \le t} dU_i \left(a_i s, F_i(x) \right).
\end{eqnarray*}
Here, $W^{i0}=\{W^{i0}(x): x \in [0,1]\}$ is a Brownian bridge,
$\mathbb{G}_i$ is the $i-$th component of the $k-$dimensional Gaussian process $\mathbb{G}$ given in Theorem~\ref{thm:multi-FCLT}, $U_i$ is a
Kiefer process which is a two-parameter continuous centered Gaussian process on $\mathbb{R}_+ \times [0,1]$ with covariance function
\begin{equation*}
\mathbb{E}[ U_i(s, x) U_i(t,y) ]= (s \wedge t ) (x \wedge y - xy).
\end{equation*}
For all $1 \le i \le k$, all the random elements
$\xi_i, W^{i0}, U_i$ are mutually independent, and they are independent of $ \mathbb{G}$.
\end{proposition}


Note that in the above result, the weak convergence of the component process $\mathbb{X}_i^{\mu}$ to $\mathbb{X}_i$ for each fixed $i$ follows directly from our Theorem~\ref{thm:multi-FCLT} and Theorem~3 in \cite{Kri1997}. However, we still need to show the joint weak convergence of the sequence $(\mathbb{X}_1^{\mu}, \ldots, \mathbb{X}_k^{\mu})$ as $\mu \rightarrow \infty$. We provide a proof in Appendix~\ref{sec:proof-13}.

One can also readily see that the limit process $\mathbb{X}$ in Proposition~\ref{thm:IS-multi-gen} is a $k$-dimensional Gaussian process given its initial state. The covariance function of $\mathbb{X}$ can be computed in a similar manner as we have done in Section~\ref{sec:property} in the one--dimensional case. For notational simplicity and illustration purposes, we study this limiting Gaussian process in detail for the special case of exponential service times in the following section.

\subsection{An example: exponential service times}

In this section, we consider the special case that class $i$ customers have the mean
service requirement $1/r_i$ with $r_i>0$, and the service time distributions are independent exponentials for all $i=1, \ldots, k$. That is,
\begin{equation} \label{eq:exp-multi}
F_i(x)= F_{i0}(x) = 1- e^{-r_i x},  \quad \text{for $x \ge 0$, and for each $i=1, \ldots, k$.}
\end{equation}
Then we immediately obtain the following result from Proposition~\ref{thm:IS-multi-gen} and Part II of Theorem~3 in \cite{Kri1997}. The proof is omitted.

\begin{proposition} \label{thm:IS-mutual}
Suppose Assumption~1 and \eqref{eq:exp-multi} holds. Assume \eqref{eq:initial} with $q_{i0} = a_i/r_i$ for $i=1, \ldots, k$.
Then as $\mu \rightarrow \infty,$ the sequence of processes
$\mathbb{X}^{\mu}$ in \eqref{eq:X-Q} converges in distribution to the process $\mathbb{X} = (\mathbb{X}_1, \ldots, \mathbb{X}_k)$ with continuous sample paths in $(D([0,\infty),\mathbb{R}^k),J_{1})$ and for $i=1, \ldots, k$,
\begin{eqnarray} \label{eq:X-multi}
\mathbb{X}_i(t) = \xi_i - r_i \int_{0}^t \mathbb{X}_i (s) ds + \mathbb{G}_i(t) + \sqrt{a_i} \cdot \mathbb{B}_i (t),
\end{eqnarray}
where $\mathbb{G}= (\mathbb{G}_1, \ldots, \mathbb{G}_k) $ is the mean-zero Gaussian process given in Theorem~\ref{thm:multi-FCLT}, $\mathbb{B}= (\mathbb{B}_1, \ldots, \mathbb{B}_k) $ is a standard $k-$dimensional Brownian motion, and $\mathbb{X}(0)= (\xi_i)_{1 \le i \le k}$, $\mathbb{G}$ and $\mathbb{B}$ are mutually independent.
\end{proposition}

Proposition~\ref{thm:IS-mutual} suggests that given $\mathbb{X}(0) \in \mathbb{R}^k$, the limit process $\mathbb{X}$ can be viewed as a $k$-dimensional Gaussian--driven OU process. 
We next provide a characterization of this multi--dimensional Gaussian--driven OU process $\mathbb{X}$ by computing its covariance function and steady--state distribution explicitly.

Given $\mathbb{X}(0) \in \mathbb{R}^k$, for $0 \le s \le t$, we write
\begin{equation*}
\text{Cov}(\mathbb{X}(t), \mathbb{X}(s) ) =(\text{Cov}(\mathbb{X}_{i}(t),\mathbb{X}_{j}(s)))_{1\leq i,j\leq k}.
\end{equation*}
To compute this covariance function, we note from \eqref{eq:X-multi} that for each $i,$
\begin{equation*}
\mathbb{X}_{i}(t)=\mathbb{X}_{i}(0)e^{-r_{i}t}
+\sqrt{a_{i}}e^{-r_{i}t}\int_{0}^{t}e^{r_{i}s}d\mathbb{B}_{i}(s)
+e^{-r_{i}t}\int_{0}^{t}e^{r_{i}s}d\mathbb{G}_{i}(s).
\end{equation*}
Therefore, we can compute that
\begin{align*}
\text{Cov}(\mathbb{X}_{i}(t),\mathbb{X}_{j}(s))
&=\text{Cov}\left(\sqrt{a_{i}}\int_{0}^{t}e^{r_{i}(u-t)}d\mathbb{B}_{i}(u),
\sqrt{a_{j}}\int_{0}^{s}e^{r_{j}(v-s)}d\mathbb{B}_{j}(v)\right)
\\
&\qquad
+\text{Cov}\left(\int_{0}^{t}e^{r_{i}(u-t)}d\mathbb{G}_{i}(u),
\int_{0}^{s}e^{r_{j}(v-s)}d\mathbb{G}_{j}(v)\right).
\end{align*}
We can compute that
\begin{align*}
&\text{Cov}\left(\sqrt{a_{i}}\int_{0}^{t}e^{r_{i}(u-t)}d\mathbb{B}_{i}(u),
\sqrt{a_{j}}\int_{0}^{s}e^{r_{j}(v-s)}d\mathbb{B}_{j}(v)\right)
\\
&=1_{i=j}\cdot a_{i}e^{-r_{i}(t+s)}\text{Var}\left(\int_{0}^{s}e^{r_{i}v}d\mathbb{B}_{i}(v)\right)
=1_{i=j}\cdot \frac{a_{i}}{2r_{i}}[e^{-r_{i}(t-s)}-e^{-r_{i}(t+s)}],
\end{align*}
and similar as in the univariate Hawkes process case,
\begin{align*}
&\text{Cov}\left(\int_{0}^{t}e^{r_{i}(u-t)}d\mathbb{G}_{i}(u),
\int_{0}^{s}e^{r_{j}(v-s)}d\mathbb{G}_{j}(v)\right)
\\
&=
\int_{0}^{s}e^{-r_{i}(t-u)}e^{-r_{j}(s-u)}dK_{ij}(u)
+\int_{0}^{s}\int_{s}^{t}e^{-r_{i}(t-u)}e^{-r_{j}(s-v)}\Phi_{ij}(v-u)dvdu
\\
&=1_{i=j}\cdot a_{i}\int_{0}^{s}e^{-r_{i}(t-u)}e^{-r_{i}(s-u)}du
+
2\int_{0}^{s}\int_{0}^{v}e^{-r_{i}(t-u)}e^{-r_{j}(s-v)}\Phi_{ij}(u-v)dvdu
\\
&\qquad\qquad
+\int_{0}^{s}\int_{s}^{t}e^{-r_{i}(t-u)}e^{-r_{j}(s-v)}\Phi_{ij}(v-u)dvdu
\\
&=1_{i=j}\cdot \frac{a_{i}}{2r_{i}}[e^{-r_{i}(t-s)}-e^{-r_{i}(t+s)}]
+\int_{0}^{s}\int_{0}^{t}e^{-r_{i}(t-u)}e^{-r_{j}(s-v)}\Phi_{ij}(v-u)dvdu,
\end{align*}
by using the definition of $K_{ij}(u)$ in \eqref{var-multi}.

Hence, we conclude that for $1 \le i,j \le k$ and $0 \le s \le t,$
\begin{align}\label{eq:cov-kd}
&\text{Cov}(\mathbb{X}_{i}(t),\mathbb{X}_{j}(s))
\\
&=1_{i=j}\cdot \frac{a_{i}}{r_{i}}[e^{-r_{i}(t-s)}-e^{-r_{i}(t+s)}]
+\int_{0}^{s}\int_{0}^{t}e^{-r_{i}(t-u)}e^{-r_{j}(s-v)}\Phi_{ij}(v-u)dvdu.
\nonumber
\end{align}
In addition, it readily follows from \eqref{eq:cov-kd} that
\begin{equation*}
\text{Cov}(\mathbb{X}_{i}(t),\mathbb{X}_{j}(t))
=1_{i=j}\cdot \frac{a_{i}}{r_{i}}[1-e^{-2r_{i}t}]
+\int_{0}^{t}\int_{0}^{t}e^{-r_{i}u}e^{-r_{j}v}\Phi_{ij}(v-u)dvdu,
\end{equation*}
where $\Phi_{ij}$ is given in \eqref{cov-multi}.
As $t \rightarrow \infty$, the sequence of random vectors $\mathbb{X}(t)$ converges in distribution to a limiting $k-$dimensional Gaussian random vector $\mathbb{X}(\infty)$ which has mean zero and covariance
\begin{equation*}
\text{Cov}(\mathbb{X}_{i}(\infty),\mathbb{X}_{j}(\infty))
=1_{i=j}\cdot \frac{a_{i}}{r_{i}}
+\int_{0}^{\infty}\int_{0}^{\infty}e^{-r_{i}u}e^{-r_{j}v}\Phi_{ij}(v-u)dvdu.
\end{equation*}
Hence, we have obtained the covariance function and the steady--state distribution of the multi--dimensional Gaussian--driven OU process $\mathbb{X}$ in \eqref{eq:X-multi}.

\section*{Acknowledgements}
We are grateful to two anonymous referees, and the Associate Editor
for very careful readings of the manuscript, and helpful suggestions,
that greatly improve the quality of the paper.
We also thank Jim Dai for helpful comments and Junfei Huang for many useful discussions.
Xuefeng Gao acknowledges support from Hong Kong RGC ECS Grant 24207015 and CUHK Direct Grants for Research with project codes 4055035 and 4055054. Lingjiong Zhu is grateful to the support from NSF Grant DMS-1613164.

\clearpage

\appendix

\section{Proofs of results in Section~\ref{sec:2}} \label{app:proof-1}
This section collects the proofs of results in Section~\ref{sec:2}.

\subsection{Proof of Theorem~\ref{thm:FCLT}} \label{sec:proof1}
\begin{proof}[Proof of Theorem~\ref{thm:FCLT}]The proof relies on Hahn's theorem (see Theorem 2 in \cite{Hahn} or Theorem 7.2.1. in \cite{whitt2002}), and delicate estimates of moments of stationary Hawkes processes.

For the sake of simplicity, we first consider that $\mu$ is a positive integer.
By the immigration-birth representation, we can decompose $N^{\mu}$
as the sum of $\mu$ independent and identically distributed (i.i.d) Hawkes processes $N_i^1, i =1, 2,\ldots, \mu$, each
distributed as a stationary Hawkes process with baseline intensity $1$ (the superscript 1 in $N_i^1$) and the exciting function $h(\cdot)$. For notational simplicity, we use $N_i(\cdot)$ for $N_i^1 (\cdot)$.
Therefore, we have
\begin{equation*}
\frac{N^{\mu}(t)-\bar \lambda t}{\sqrt{\mu}}
=\frac{1}{\sqrt{\mu}}\sum_{i=1}^{\mu}\left[N_i ({t})-\frac{t}{1-\Vert h\Vert_{L^{1}}}\right].
\end{equation*}

Let $\tilde{N}_i ({t}):=N_i({t})-\frac{t}{1-\Vert h\Vert_{L^{1}}}$.
Then, $\tilde{N}_i$ are i.i.d. random elements of $D([0,\infty),\mathbb{R})$
with $\mathbb{E}[\tilde{N}_i(t)]=0$ (see e.g. Equation (9) in \cite{Hawkes}) and
$\mathbb{E}[(\tilde{N}_i (t))^{2}]<\infty$ for any $t$ (see e.g. Lemma 2 in \cite{ZhuCLT}\footnote{In Lemma 2 in \cite{ZhuCLT}, it was proved that $\mathbb{E}[(N_{i}(1))^{2}]<\infty$.
By the stationarity of $N_i$ and the Cauchy-Schwarz inequality, for every positive integer $t$,
$\mathbb{E}[(N_{i}(t))^{2}]=\mathbb{E}[(\sum_{j=1}^{t}N_{i}(j-1,j))^{2}]
\leq t\sum_{j=1}^{t}\mathbb{E}[(N_{i}(j-1,j))^{2}]=t^{2}\mathbb{E}[(N_{i}(1))^{2}]<\infty$.}).

By Hahn's theorem, \textit{since $\tilde{N}_{i}$ are i.i.d.},
as $\mu \rightarrow \infty,$ we have
\begin{equation*}
\frac{1}{\sqrt{\mu}}\sum_{i=1}^{\mu}\left[N_i ({t})-\frac{t}{1-\Vert h\Vert_{L^{1}}}\right]
=\frac{1}{\sqrt{\mu}}\sum_{i=1}^{\mu}\tilde{N}_{i}(t)
\Rightarrow G(t),
\end{equation*}
weakly in $(D([0,\infty),\mathbb{R}),J_{1})$, where $G$ is a mean-zero almost surely continuous Gaussian process
with the covariance function of $\tilde{N}_1$ provided that \textit{the following
condition is satisfied}: For every $0<T<\infty$, there exist continuous nondecreasing real-valued
functions $g$ and $f$ on $[0,T]$ with numbers $\alpha>1/2$ and $\beta>1$ such that
\begin{equation}\label{ConI}
\mathbb{E}\left[\left(\tilde{N}_1 ({u}) -\tilde{N}_1 ({s}) \right)^{2}\right]\leq(g(u)-g(s))^{\alpha},
\end{equation}
and \begin{equation}\label{ConII}
\mathbb{E}\left[\left(\tilde{N}_1 ({u}) - \tilde{N}_1 ({t})\right)^{2}
\left(\tilde{N}_1 ({t}) -\tilde{N}_1 ({s})\right)^{2}\right]\leq(f(u)-f(s))^{\beta},
\end{equation}
for all $0\leq s\leq t\leq u\leq T$ with $u-s<1$.

Let us prove \eqref{ConI} and \eqref{ConII}. For notational simplicity, we use $N_1(a,b]$ to stand for $N_1((a,b])$ (equivalently, $N_1(b)-N_1(a)$) which records the number of points of the process $N_1$ in the interval $(a,b]$. We also use $\lambda_1$ to denote the intensity process of the stationary Hawkes process $N_1$ with baseline intensity 1. We now present a lemma which is the key to the proofs of \eqref{ConI} and \eqref{ConII}. The proof is given at the end of this section.

\begin{lemma} \label{lemma:1}
We have
\begin{equation} \label{eq:intensity-moment}
\mathbb{E}[(\lambda_1(0))^4] < \infty.
\end{equation}
As a result, for all $0\leq s\leq t\leq u\leq T$ and $u-s<1$, there are some constants $c, C>0$ independent of $s, t, u,$ such that
\begin{eqnarray}
\mathbb{E}\left[(N_1(s,u])^{2}\right]
&\leq & C \cdot (u-s), \label{eq:inter-step0} \\
\mathbb{E}\left[(N_1(t,u])^{2}(N_1(s,t])^{2}\right] &\le& c \cdot (u-s)^2, \label{eq:inter-step}
\end{eqnarray}
\end{lemma}

With Lemma~\ref{lemma:1}, we are ready to prove \eqref{ConI} and \eqref{ConII}.
First, let us prove \eqref{ConI}. It is clear from the definition of $\tilde N_1$ that
\begin{align*}
\mathbb{E}\left[\left(\tilde{N}_1 ({u}) -\tilde{N}_1 ({s}) \right)^{2}\right]
&=\mathbb{E}\left[\left(N_1(s,u]-\frac{u-s}{1-\Vert h\Vert_{L^{1}}}\right)^{2}\right]
\\
&\leq\mathbb{E}\left[\left(N_1(s,u]\right)^{2}\right]
+\left(\frac{u-s}{1-\Vert h\Vert_{L^{1}}}\right)^{2}.
\nonumber
\end{align*}
Using \eqref{eq:inter-step0} in Lemma~\ref{lemma:1} and the fact that $0  \le u-s<1$,
we immediately obtain that
\eqref{ConI} is satisfied with $g(x)=\left(C + \frac{1}{(1- \Vert h\Vert_{L^{1}})^2} \right) x$ and $\alpha=1.$

Next, let us prove \eqref{ConII}. Note that
\begin{align} \label{eq:conII}
&\mathbb{E}\left[\left(\tilde{N}_1 ({u}) - \tilde{N}_1 ({t})\right)^{2}
\left(\tilde{N}_1 ({t}) -\tilde{N}_1 ({s})\right)^{2}\right]
\\
&=\mathbb{E}\left[\left(N_1(t,u]-\frac{u-t}{1-\Vert h\Vert_{L^{1}}}\right)^{2}
\left(N_1[s,t]-\frac{t-s}{1-\Vert h\Vert_{L^{1}}}\right)^{2}\right]
\nonumber
\\
&\leq
\mathbb{E}\left[\left((N_1 (t,u])^{2}+\left(\frac{u-t}{1-\Vert h\Vert_{L^{1}}}\right)^{2}\right)
\left((N_1 (s,t])^{2}+\left(\frac{t-s}{1-\Vert h\Vert_{L^{1}}}\right)^{2}\right)\right]
\nonumber
\\
&=
\left(\frac{u-t}{1-\Vert h\Vert_{L^{1}}}\right)^{2}\mathbb{E}\left[(N_1(s,t])^{2}\right]
+\left(\frac{t-s}{1-\Vert h\Vert_{L^{1}}}\right)^{2}\mathbb{E}\left[(N_1(t,u])^{2}\right]
\nonumber
\\
&\qquad\qquad
+\left(\frac{u-t}{1-\Vert h\Vert_{L^{1}}}\right)^{2}\left(\frac{t-s}{1-\Vert h\Vert_{L^{1}}}\right)^{2}
+\mathbb{E}\left[(N_1 (s,t])^{2}(N_1(t,u])^{2}\right].
\nonumber
\end{align}
Since $0\leq s\leq t\leq u\leq T$ and $u-s<1$, we can then infer from Lemma~\ref{lemma:1} and \eqref{eq:conII} that \eqref{ConII} is satisfied with $f(x)=C'x$ for some positive constant $C'$ (independent of $u, s , t$) and $\beta=2$.

Now we have proved Theorem~\ref{thm:FCLT} by assuming $\mu$ is a positive integer in our discussions.
The same result holds when $\mu\in (0, \infty)$ for $\mu\rightarrow\infty$. Note that for $\mu\in(0, \infty)$, by the immigration-birth representation,
the process $N^{\mu}(\cdot)$ can be decomposed as the sum of two independent stationary Hawkes processes $N^{\lfloor\mu\rfloor}(\cdot)$
and $N^{\mu-\lfloor\mu\rfloor}(\cdot)$, where the superscripts $\lfloor\mu\rfloor, \mu-\lfloor\mu\rfloor$ represent the baseline intensities, respectively. Hence,
to show the result holds for $\mu\in (0, \infty)$ with $\mu \rightarrow\infty$, it suffices to show that for any $T>0$,
\begin{equation*}
\sup_{0\leq t\leq T}\frac{N^{\mu-\lfloor\mu\rfloor}(t)-\frac{(\mu-\lfloor\mu\rfloor)t}{1-\Vert h\Vert_{L^{1}}}}{\sqrt{\mu}}\rightarrow 0,
\end{equation*}
in probability as $\mu\rightarrow\infty$. This can be easily verified since
for any $\epsilon>0$, for sufficiently large $\mu$, we have $\sup_{0\leq t\leq T}\frac{1}{\sqrt{\mu}}
\frac{(\mu-\lfloor\mu\rfloor)t}{1-\Vert h\Vert_{L^{1}}}
\leq\frac{1}{\sqrt{\mu}}\frac{T}{1-\Vert h\Vert_{L^{1}}}\leq\frac{\epsilon}{2}$, and
\begin{align*}
\mathbb{P}\left(\left|\sup_{0\leq t\leq T}\frac{N^{\mu-\lfloor\mu\rfloor}(t)-\frac{(\mu-\lfloor\mu\rfloor)t}{1-\Vert h\Vert_{L^{1}}}}{\sqrt{\mu}}\right|\geq\epsilon\right)
&\leq
\mathbb{P}\left(\left|\sup_{0\leq t\leq T}\frac{N^{\mu-\lfloor\mu\rfloor}(t)}{\sqrt{\mu}}\right|\geq\frac{\epsilon}{2}\right)
\\
&=\mathbb{P}\left(N^{\mu-\lfloor\mu\rfloor}(T)\geq\frac{1}{2}\sqrt{\mu}\epsilon\right)
\nonumber
\\
& \leq\mathbb{P}\left(N^{1}(T)\geq\frac{1}{2}\sqrt{\mu}\epsilon\right)\leq\frac{2}{\sqrt{\mu}\epsilon}T \cdot \mathbb{E}[N^{1}(1)]\rightarrow 0,
\nonumber
\end{align*}
as $\mu\rightarrow\infty$.
Here $N^1$ denotes a stationary Hawkes process with a baseline intensity one and an exciting function $h$.

Finally, let us compute the covariance function of $G$, or equivalently (from Hahn's Theorem), the covariance function of $\tilde{N}_1(\cdot)$.
Since $\tilde{N}_1 (t)$ and $\tilde{N}_1 (s)$ have mean zero,
we can compute that, for any $t>s$,
\begin{align}\label{CovCompute1}
\mbox{Cov}(\tilde{N}_1(t),\tilde{N}_1(s))
&=\mathbb{E}\left[\tilde{N}_1(t) \tilde{N}_1(s)\right]
\\
&=\mathbb{E}\left[\left(N_1(t)-\frac{t}{1-\Vert h\Vert_{L^{1}}}\right)
\left(N_1(s)-\frac{s}{1-\Vert h\Vert_{L^{1}}}\right)\right]
\nonumber
\\
&=\mathbb{E}[N_1 (t) N_1(s)]
-\frac{ts}{(1-\Vert h\Vert_{L^{1}})^{2}}
\nonumber
\\
&=\mathbb{E}[(N_1(t)-N_1(s)) N_1(s)]
+\mathbb{E}[(N_1(s))^{2}]-\frac{ts}{(1-\Vert h\Vert_{L^{1}})^{2}}.
\nonumber
\end{align}
It is clear that
\begin{equation*}\label{CovCompute2}
\mathbb{E}[(N_1(s))^{2}]
=\mbox{Var}(N_1(s))+
\frac{s^{2}}{(1-\Vert h\Vert_{L^{1}})^{2}}
=K(s)+\frac{s^{2}}{(1-\Vert h\Vert_{L^{1}})^{2}}.
\end{equation*}
In addition, we can verify that
\begin{eqnarray*}\label{CovCompute3}
\mathbb{E}[(N_1(t)-N_1(s))N_1(s)]
&=&\mathbb{E}\left[\int_{s+}^{t}N_1(du)\int_{0}^{s}N_1(dv)\right] =\int_{s+}^{t}\int_{0}^{s}\mathbb{E}[N_1(dv) N_1(du)]
\nonumber
\\
&=&\int_{s}^{t}\int_{0}^{s}\left[\phi(u-v)+\frac{1}{(1-\Vert h\Vert_{L^{1}})^{2}}\right]dvdu
\nonumber
\\
&=&\int_{s}^{t}\int_{0}^{s}\phi(u-v)dvdu+\frac{s(t-s)}{(1-\Vert h\Vert_{L^{1}})^{2}}.
\end{eqnarray*}
Hence, we get
\begin{align}\label{CovCompute4}
& \mbox{Cov}(G(t),G(s)) = \mbox{Cov}(\tilde{N}_1(t),\tilde{N}_1(s))
\\
&=\int_{s}^{t}\int_{0}^{s}\phi(u-v)dudv+\frac{s(t-s)}{(1-\Vert h\Vert_{L^{1}})^{2}}
+K(s)+\frac{s^{2}}{(1-\Vert h\Vert_{L^{1}})^{2}}-\frac{ts}{(1-\Vert h\Vert_{L^{1}})^{2}}
\nonumber
\\
&=\int_{s}^{t}\int_{0}^{s}\phi(u-v)dudv+K(s).
\nonumber
\end{align}
The proof is therefore complete.
\end{proof}

\begin{proof}[Proof of Lemma~\ref{lemma:1}]
We first prove \eqref{eq:intensity-moment}.  Using the definition of the intensity $\lambda_1$ and the simple inequality $\left(\frac{x+y}{2}\right)^4 \le \frac{x^4 + y^4}{2}$, we obtain that for sufficiently small $\delta>0$,
\begin{eqnarray}
\mathbb{E}\left[(\lambda_{1}(0))^{4}\right]
&=&\mathbb{E}\left[\left(1+\int_{-\infty}^{0}h(-s)N_1(ds)\right)^{4}\right]
\nonumber
\\
&\leq&
8+8\mathbb{E}\left[\left(\int_{-\infty}^{0}h(-s)N_1(ds)\right)^{4}\right]
\nonumber\\
&\leq&
8+8\mathbb{E}\left[\left(\sum_{i=0}^{\infty}\max_{t\in[-(i+1)\delta,-i\delta]}h(-t)\cdot N_1[-(i+1)\delta,-i\delta]\right)^{4}\right] \nonumber \\
& = &
8+8\left(\sum_{i=0}^{\infty}\max_{t\in[-(i+1)\delta,-i\delta]}h(-t)\right)^{4} \times \nonumber \\
&&\mathbb{E}\left[\left(\sum_{i=0}^{\infty}\frac{\max_{t\in[-(i+1)\delta,-i\delta]}h(t)}{\sum_{i=0}^{\infty}\max_{t\in[-(i+1)\delta,-i\delta]}h(t)}N_1[-(i+1)\delta,-i\delta]\right)^{4}\right].
\label{eq:tempstep}
\end{eqnarray}
Note here that under Assumption~1, we know $h(\cdot)$ is locally bounded and Riemann integrable, hence for sufficiently small $\delta>0$,
\begin{equation} \label{eq:h-sumbound}
\sum_{i=0}^{\infty}\max_{t\in[-(i+1)\delta,-i\delta]}h(-t) < \infty.
\end{equation}
Applying the Jensen's inequality to \eqref{eq:tempstep},
we get
\begin{align}\label{eq:3}
&\mathbb{E}\left[(\lambda_{1}(0))^{4}\right]
\\
&\leq
8+8\left(\sum_{i=0}^{\infty}\max_{t\in[-(i+1)\delta,-i\delta]}h(-t)\right)^{4}
\mathbb{E}\left[\sum_{i=0}^{\infty}\frac{\max_{t\in[-(i+1)\delta,-i\delta]}h(t)}{\sum_{i=0}^{\infty}\max_{t\in[-(i+1)\delta,-i\delta]}h(t)}\left(N_1[-(i+1)\delta,-i\delta]\right)^{4}\right]
\nonumber
\\
&=
8+8\left(\sum_{i=0}^{\infty}\max_{t\in[-(i+1)\delta,-i\delta]}h(-t)\right)^{4}
\mathbb{E}\left[\left(N_1[0,\delta]\right)^{4}\right]<\infty,
\nonumber
\end{align}
where we have used the stationarity of $N_1$, \eqref{eq:h-sumbound} and the fact that $\mathbb{E}[e^{\theta N_1[0,1]}]<\infty$
for sufficiently small $\theta>0$, see e.g. \cite{ZhuCLT}.

We next prove \eqref{eq:inter-step0}. We can directly compute that
\begin{align}\label{C1estimate}
\mathbb{E}\left[\left(N_1(s,u]\right)^{2}\right]
&\leq 2\mathbb{E}\left[\left(N_{1}(s,u]-\int_{s}^{u}\lambda_{1}(v)dv\right)^{2}\right]
+2\mathbb{E}\left[\left(\int_{s}^{u}\lambda_{1}(v)dv\right)^{2}\right]
\\
&=
2\mathbb{E}\left[\int_{s}^{u}\lambda_{1}(v)dv\right]
+2\mathbb{E}\left[\left(\int_{s}^{u}\lambda_{1}(v)dv\right)^{2}\right]
\nonumber
\\
&=2\frac{(u-s)}{1-\Vert h\Vert_{L^{1}}}
+2\mathbb{E}\left[\left(\int_{s}^{u}\lambda_{1}(v)dv\right)^{2}\right]
\nonumber
\\
&\leq
2\frac{(u-s)}{1-\Vert h\Vert_{L^{1}}}
+2(u-s)\mathbb{E}\left[\left(\int_{s}^{u}(\lambda_{1}(v))^{2}dv\right)\right]
\nonumber
\\
&=2\frac{(u-s)}{1-\Vert h\Vert_{L^{1}}}
+2(u-s)^{2}\mathbb{E}[(\lambda_{1}(0))^{2}]\nonumber \\ \
&\leq C (u-s),
\nonumber
\end{align}
for some positive constant $C$.
Here, the second line follows from the martingale property,
see Section \ref{sec:hawkes};
the third line follows from the stationarity of the intensity process
and $\mathbb{E}[\lambda_{1}(0)]=\frac{1}{1-\Vert h\Vert_{L^{1}}}$;
the fourth line follows from the Cauchy-Schwarz inequality so that
$\left(\int_{s}^{u}\lambda_1(v)dv\right)^{2}\leq\int_{s}^{u}1dv\cdot\int_{s}^{u}(\lambda_{1}(v))^{2}dv$; the fifth line
is due to the fact that $\lambda_1$ is a stationary process; and the last inequality is due to \eqref{eq:intensity-moment} and the fact that $0<u-s<1$.

Finally, we prove \eqref{eq:inter-step}.
We can directly compute that
\begin{align} \label{eq:1}
&\mathbb{E}\left[(N_1(t,u])^{2}(N_1(s,t])^{2}\right]
\\
&\leq 2\mathbb{E}\left[\left(N_1(t,u]-\int_{t}^{u}\lambda_1 (v)dv\right)^{2}(N_1(s,t])^{2}\right]
+2\mathbb{E}\left[\left(\int_{t}^{u}\lambda_1(v)dv\right)^{2}(N_1(s,t])^{2}\right]
\nonumber
\\
&\leq 2\mathbb{E}\left[\int_{t}^{u}\lambda_1 (v)dv\cdot (N_1 (s,t])^{2}\right]
+2(u-t)\mathbb{E}\left[\int_{t}^{u}(\lambda_1 (v))^{2}dv\cdot (N_1 (s,t])^{2}\right],
\nonumber
\end{align}
where we use the fact that $N_1(t,u]-\int_{t}^{u}\lambda_1 (v)dv$
is $\mathcal{F}_{t}$-measurable
and is a martingale with predictable quadratic variation $\int_{t}^{u}\lambda_1 (v)dv$
so that $\left(N_1(t,u]-\int_{t}^{u}\lambda_1 (v)dv\right)^{2}-\int_{t}^{u}\lambda_1 (v)dv$
is also a martingale (see Section \ref{sec:hawkes})
and the Cauchy-Schwarz inequality so that
$\left(\int_{t}^{u}\lambda_1(v)dv\right)^{2}\leq\int_{t}^{u}1dv\cdot\int_{t}^{u}(\lambda_{1}(v))^{2}dv$.
From here, we can further estimate that
\begin{align}\label{eq:221}
&\mathbb{E}\left[(N_1(t,u])^{2}(N_1(s,t])^{2}\right]
\\
&\leq 2\left(\mathbb{E}\left[\left(\int_{t}^{u}\lambda_1 (v) dv\right)^{2}\right]\right)^{1/2}
\left(\mathbb{E}\left[(N_1(s,t])^{4}\right]\right)^{1/2}
\nonumber
\\
&\qquad\qquad
+(u-t)\mathbb{E}\left[\int_{t}^{u}\left((\lambda_1 (v))^{4}+(N_1 (s,t])^{4}\right)dv\right],
\nonumber
\end{align}
where we applied Cauchy-Schwarz inequality to the first term
in the last line in \eqref{eq:1}
so that $\mathbb{E}\left[\int_{t}^{u}\lambda_1 (v)dv\cdot (N_1 (s,t])^{2}\right]
\leq\left(\mathbb{E}\left[\left(\int_{t}^{u}\lambda_1 (v) dv\right)^{2}\right]\right)^{1/2}
\left(\mathbb{E}\left[(N_1(s,t])^{4}\right]\right)^{1/2}$,
and the simple inequality $2(\lambda_1 (v))^{2}(N_1 (s,t])^{2}\leq(\lambda_1 (v))^{4}+(N_1 (s,t])^{4}$
to the second term in the last line in \eqref{eq:1}.
From here, we can continue that
\begin{align}\label{22}
&\mathbb{E}\left[(N_1(t,u])^{2}(N_1(s,t])^{2}\right]
\\
&\leq 2(u-t)^{1/2}\left(\left[\int_{t}^{u}\mathbb{E}[(\lambda_1(v))^{2}]dv\right]\right)^{1/2}
\left(\mathbb{E}\left[(N_1 (s,t])^{4}\right]\right)^{1/2}
\nonumber
\\
&\qquad\qquad
+(u-t)\left[\int_{t}^{u}\left(\mathbb{E}(\lambda_1(v))^{4}+\mathbb{E}(N_1 (s,t])^{4}\right)dv\right]
\nonumber
\\
&=
2(u-t)\left(\mathbb{E}[(\lambda_1(0))^{2}]\right)^{1/2}
\left(\mathbb{E}\left[(N_1(s,t])^{4}\right]\right)^{1/2}
\nonumber
\\
&\qquad\qquad
+(u-t)^{2}\mathbb{E}\left[(\lambda_1(0))^{4}\right]
+(u-t)^{2}\mathbb{E}\left[(N_1(s,t])^{4}\right],
\nonumber
\end{align}
where we used Cauchy-Schwartz inequality $\left(\int_{t}^{u}\lambda_1(v)dv\right)^{2}\leq\int_{t}^{u}1dv\cdot\int_{t}^{u}(\lambda_{1}(v))^{2}dv$,
and the stationarity of the Hawkes processes,
see Section \ref{sec:hawkes}.

Hence, to bound $\mathbb{E}\left[(N_1(t,u])^{2}(N_1(s,t])^{2}\right]$, we need to estimate $\mathbb{E}\left[(N_1 (s,t])^{4}\right]$.
We can compute that (explanations follow below)
\begin{align}\label{eq:2}
\mathbb{E}\left[(N_1(s,t])^{4}\right]
&\leq 8\mathbb{E}\left[\left(N_1 (s,t]-\int_{s}^{t}\lambda_1 (v) dv\right)^{4}\right]
+8\mathbb{E}\left[\left(\int_{s}^{t}\lambda_1(v)dv\right)^{4}\right]
\\
&\leq
8\bar{C}\mathbb{E}\left[\left(\int_{s}^{t}\lambda_1(v)dv\right)^{2}\right]
+8\mathbb{E}\left[\left(\int_{s}^{t}\lambda_1(v) dv\right)^{4}\right]
\nonumber
\\
&\leq
8\bar{C}(t-s)\mathbb{E}\left[\int_{s}^{t}(\lambda_1(v))^{2}dv\right]
+8(t-s)^{3}\mathbb{E}\left[\int_{s}^{t}(\lambda_1(v))^{4}dv\right]
\nonumber
\\
&=
8\bar{C}(t-s)^{2}\mathbb{E}\left[(\lambda_1(0))^{2}\right]
+8(t-s)^{4}\mathbb{E}\left[(\lambda_1(0))^{4}\right].
\nonumber
\end{align}
The first inequality in \eqref{eq:2} uses the inequality $(\frac{x+y}{2})^{4}\leq\frac{x^{4}+y^{4}}{2}$.
The second inequality in \eqref{eq:2} uses the martingality of $N_1 [s,t]-\int_{s}^{t}\lambda_1 (v) dv$
with the predictable quadratic variation $\int_{s}^{t}\lambda_1 (v) dv$ and the Burkholder-Davis-Gundy inequality,
where $\bar{C}>0$ is a constant from the Burkholder-Davis-Gundy inequality
\footnote{The Burkholder-Davis-Gundy inequality reads that
for a local martingale $M_{t}$ starting at $0$ at $t=0$, and $M_{t}^{\ast}:=\sup_{0\leq s\leq t}|M_{s}|$, we have
$c_{p}\mathbb{E}[\langle M\rangle_{t}^{p/2}]
\leq\mathbb{E}[(M_{t}^{\ast})^{p}]
\leq C_{p}\mathbb{E}[\langle M\rangle_{t}^{p/2}]$, for some constant
$c_{p}<C_{p}$ depending on $p\geq 1$ only
and $\langle M\rangle_{t}$ is the (predictable) quadratic variation of $M_{t}$.
As a corollary, we have $\mathbb{E}[|M_{t}|^{p}]\leq C_{p}\mathbb{E}[\langle M\rangle_{t}^{p/2}]$.
In our application, $p=4$ and $M_{t}=N_{1}[s,t]-\int_{s}^{t}\lambda_{1}(v)dv$, $t\geq s$
so that $M_{t}$ is a martingale with $M_{s}=0$ and the predictable
quadratic variation $\int_{s}^{t}\lambda_{1}(v)dv$.}.
The third inequality in \eqref{eq:2} uses Jensen's inequality so that
\begin{equation*}
\left(\frac{1}{t-s}\int_{s}^{t}\lambda_1(v) dv\right)^{2}\leq\frac{1}{t-s}\int_{s}^{t}(\lambda_1(v))^{2}dv,
\end{equation*}
and
\begin{equation*}
\left(\frac{1}{t-s}\int_{s}^{t}\lambda_1(v) dv\right)^{4}\leq\frac{1}{t-s}\int_{s}^{t}(\lambda_1(v))^{4}dv.
\end{equation*}
Finally, the last equality in \eqref{eq:2} is due to stationarity of the intensity process,
see Section~\ref{sec:hawkes}.

Combining \eqref{22}, \eqref{eq:2} and \eqref{eq:intensity-moment}, we deduce that \eqref{eq:inter-step} holds. The proof is thus complete.
\end{proof}

\subsection{Proof of Proposition~\ref{prop:Kt}}\label{sec:prop-Kt}

\begin{proof}[Proof of Proposition~\ref{prop:Kt}]
We first prove Part (b), then prove Parts (a) and (c).

To prove Part (b), we recall that
\begin{equation}\label{integratephi}
\phi(t)=\frac{h(t)}{1-\Vert h\Vert_{L^{1}}}
+\int_{0}^{\infty}h(t+v)\phi(v)dv+\int_{0}^{t}h(t-v)\phi(v)dv.
\end{equation}
Denote $H(t):=\int_t^{\infty}h(s)ds$.
By integrating Equation \eqref{integratephi} at both sides, we get
\begin{align*}
\Vert\phi\Vert_{L^{1}}
&=\frac{\Vert h\Vert_{L^{1}}}{1-\Vert h\Vert_{L^{1}}}
+\int_{0}^{\infty}H(v)\phi(v)dv
+\Vert h\Vert_{L^{1}}\Vert\phi\Vert_{L^{1}}
\\
&=\frac{\Vert h\Vert_{L^{1}}}{1-\Vert h\Vert_{L^{1}}}
+\int_{0}^{M}H(v)\phi(v)dv
+\int_{M}^{\infty}H(v)\phi(v)dv
+\Vert h\Vert_{L^{1}}\Vert\phi\Vert_{L^{1}}.
\nonumber
\end{align*}
Since $\Vert h\Vert_{L^{1}}<1$, there exists some $\epsilon>0$
so that $\Vert h\Vert_{L^{1}}+\epsilon<1$. In addition,
note that $H(M)$ is decreasing in $M$ to $0$ as $M\rightarrow\infty$.
Hence, for sufficiently large $M$ we have $H(M)\leq\epsilon$. This implies
\begin{align*}
\Vert\phi\Vert_{L^{1}}
&\leq\frac{\Vert h\Vert_{L^{1}}}{1-\Vert h\Vert_{L^{1}}}
+H(0)\int_{0}^{M}\phi(v)dv
+\epsilon\Vert\phi\Vert_{L^{1}}
+\Vert h\Vert_{L^{1}}\Vert\phi\Vert_{L^{1}}.
\nonumber
\end{align*}
Note that $\text{Var}(N^{1}(t))<\infty$ (See e.g. Lemma 2 in \cite{ZhuCLT})
and hence
\begin{equation}
K(t)=\frac{t}{1-\Vert h\Vert_{L^{1}}}+2\int_{0}^{t}\int_{0}^{t_{2}}\phi(t_{2}-t_{1})dt_{1}dt_{2}<\infty.
\end{equation}
Moreover, $\int_{0}^{t_{2}}\phi(t_{2}-t_{1})dt_{1}=\int_{0}^{t_{2}}\phi(t_{1})dt_{1}$
is non-decreasing in $t_{2}$, and $\int_{0}^{t}\int_{0}^{t_{2}}\phi(t_{1})dt_{1}dt_{2}<\infty$
for every $t$. It implies that $\int_{0}^{t_{2}}\phi(t_{1})dt_{1}<\infty$ for every $t_{2}$.
Thus, $\int_{0}^{M}\phi(v)dv<\infty$.
Hence, it follows that
\begin{equation} \label{eq:phi-int}
\Vert\phi\Vert_{L^{1}}
\leq\frac{\Vert h\Vert_{L^{1}}}{(1-\Vert h\Vert_{L^{1}})(1-\Vert h\Vert_{L^{1}}-\epsilon)}
+\frac{H(0)\int_{0}^{M}\phi(v)dv}{1-\Vert h\Vert_{L^{1}}-\epsilon}<\infty.
\end{equation}

To establish the second part of Part (b), we first note from the definition of $K(t)$ in \eqref{eq:var} that
\begin{equation*}
K'(t)=\frac{1}{1-\Vert h\Vert_{L^{1}}}+2\int_{0}^{t}\phi(t-t_{1})dt_{1}.
\end{equation*}
The differentiability of $K(\cdot)$ is due to the integrability of $\phi$ as given in \eqref{eq:phi-int}. Since $\phi$ is nonnegative, it follows that $K'(\cdot)$ is non-decreasing. Hence, $K(\cdot)$ is convex. In addition, we note that $\int_{0}^{t}\phi(t-t_{1})dt_{1}=\int_{0}^{t}\phi(t_{1})dt_{1} \le \Vert \phi \Vert_{L^{1}}< \infty$ for all $t$.
Thus $K(\cdot)$ is Lipschitz continuous.

We next provide a proof of Part (a) which will be useful in proving Part (c). Write $N^1$ for a stationary Hawkes process with baseline intensity $\mu=1$.
From the Bartlett spectrum for the stationary Hawkes process (see \cite{Hawkes} or \cite{Daley}), we know that
\begin{equation*}
\mbox{Var}\left(\int_{\mathbb{R}}\psi(s)N^1(ds)\right)
=\int_{\mathbb{R}}|\hat{\psi}(\omega)|^{2}\frac{1}{2\pi(1-\Vert h\Vert_{L^{1}})}\frac{1}{|1-\hat{h}(\omega)|^{2}}d\omega,
\end{equation*}
where $\hat{\psi}(\omega) = \int_{\mathbb{R}} e^{i\omega t} \psi(t) dt$.
Note that
\begin{equation*}
K(t)=\mbox{Var}(N^1(t))
=\mbox{Var}\left(\int_{\mathbb{R}}1_{[0,t]}(s)N^1(ds)\right),
\end{equation*}
and
\begin{equation*}
\int_{\mathbb{R}}e^{i\omega s}1_{[0,t]}(s)ds
=\int_{0}^{t}e^{i\omega s}ds
=\frac{e^{i\omega t}-1}{i\omega}.
\end{equation*}
We can compute that
\begin{eqnarray*}
\left|\frac{e^{i\omega t}-1}{i\omega}\right|^{2}
=\frac{(\cos(\omega t)-1)^{2}+\sin^{2}(\omega t)}{\omega^{2}}
=\frac{2-2\cos(\omega t)}{\omega^{2}}
=\frac{\sin^{2}(\frac{1}{2}\omega t)}{(\frac{1}{2}\omega)^{2}}.
\end{eqnarray*}
Therefore,
\begin{align}
K(t)&=\frac{1}{2\pi(1-\Vert h\Vert_{L^{1}})}
\int_{\mathbb{R}}\frac{\sin^{2}(\frac{1}{2}\omega t)}{(\frac{1}{2}\omega)^{2}}
\frac{1}{|1-\hat{h}(\omega)|^{2}}d\omega
\nonumber\\
&=\frac{t}{2\pi(1-\Vert h\Vert_{L^{1}})}
\int_{\mathbb{R}}\frac{\sin^{2}(\frac{1}{2}\omega)}{(\frac{1}{2}\omega)^{2}}
\frac{1}{|1-\hat{h}(\frac{\omega}{t})|^{2}}d\omega.
\label{eq:Kt-barlett}
\end{align}
Notice that for any $t$,
\begin{equation*}
\left|1-\hat{h}\left(\frac{\omega}{t}\right)\right|
\geq 1-\left|\hat{h}\left(\frac{\omega}{t}\right)\right|
\geq 1-\Vert h\Vert_{L^{1}}>0,
\end{equation*}
and $\int_{\mathbb{R}}\frac{\sin^{2}(\frac{1}{2}\omega)}{(\frac{1}{2}\omega)^{2}}d\omega=2\pi$.
Thus, $\frac{\sin^{2}(\frac{1}{2}\omega)}{(\frac{1}{2}\omega)^{2}}
\frac{1}{|1-\hat{h}(\frac{\omega}{t})|^{2}}
\leq\frac{\sin^{2}(\frac{1}{2}\omega)}{(\frac{1}{2}\omega)^{2}}
\frac{1}{(1-\Vert h\Vert_{L^{1}})^{2}}$, which is integrable.
On the other hand, for every $\omega$, $\lim_{t\rightarrow\infty}\hat{h}(\frac{\omega}{t})
=\hat{h}(0)=\Vert h\Vert_{L^{1}}$.
Therefore, by dominated convergence theorem, we obtain
\begin{equation*}
\lim_{t\rightarrow\infty}\frac{K(t)}{t}=\frac{1}{(1-\Vert h\Vert_{L^{1}})^{3}}.
\end{equation*}

We now prove Part (c). This requires a more delicate analysis of the Bartlett spectrum. Write $\overline z$ for the complex conjugate of a complex number $z$.
We can obtain from \eqref{eq:Kt-barlett} and $\hat{h}(0)=\Vert h\Vert_{L^{1}}$ that
\begin{align}
\lefteqn{K(t)-\frac{t}{(1-\Vert h\Vert_{L^{1}})^{3}}} \nonumber \\
&=\frac{t}{2\pi(1-\Vert h\Vert_{L^{1}})}\int_{\mathbb{R}}\frac{\sin^{2}(\frac{1}{2}\omega)}{(\frac{1}{2}\omega)^{2}}
\left[\frac{1}{|1-\hat{h}(\frac{\omega}{t})|^{2}}-\frac{1}{|1-\hat{h}(0)|^{2}}\right]d\omega
\nonumber\\
&=\frac{t}{2\pi(1-\Vert h\Vert_{L^{1}})^{3}}\int_{\mathbb{R}}\frac{\sin^{2}(\frac{1}{2}\omega)}{(\frac{1}{2}\omega)^{2}} \cdot
\frac{-2\Vert h\Vert_{L^{1}}+\Vert h\Vert_{L^{1}}^{2}
+\hat{h}(\frac{\omega}{t})+\overline{\hat{h}(\frac{\omega}{t})}
-|\hat{h}(\frac{\omega}{t})|^{2}}{|1-\hat{h}(\frac{\omega}{t})|^{2}}d\omega
\nonumber
\\
&=\frac{1}{2\pi(1-\Vert h\Vert_{L^{1}})^{3}}\int_{\mathbb{R}}\frac{\sin^{2}(\frac{1}{2}\omega t)}{(\frac{1}{2}\omega)^{2}} \cdot
\frac{-2\Vert h\Vert_{L^{1}}+\Vert h\Vert_{L^{1}}^{2}
+\hat{h}(\omega)+\overline{\hat{h}(\omega)}
-|\hat{h}(\omega)|^{2}}{|1-\hat{h}(\omega)|^{2}}d\omega
\nonumber
\\
&=\frac{1}{2\pi(1-\Vert h\Vert_{L^{1}})^{3}}\int_{\mathbb{R}}\sin^{2}\left(\frac{1}{2}\omega t\right)f(\omega)d\omega
\nonumber
\\
&=\frac{1}{4\pi(1-\Vert h\Vert_{L^{1}})^{3}}\int_{\mathbb{R}}f(\omega)d\omega
-\frac{1}{4\pi(1-\Vert h\Vert_{L^{1}})^{3}}\int_{\mathbb{R}}\cos(\omega t)f(\omega)d\omega
\label{eq:Kt-linear}
\end{align}
where
\begin{equation*}
f(\omega)=\frac{1}{(\frac{1}{2}\omega)^{2}} \cdot \frac{-2\Vert h\Vert_{L^{1}}+\Vert h\Vert_{L^{1}}^{2}
+\hat{h}(\omega)+\overline{\hat{h}(\omega)}
-|\hat{h}(\omega)|^{2}}{|1-\hat{h}(\omega)|^{2}}.
\end{equation*}

We claim that $f(\omega)\in L^{1}(\mathbb{R})$, i.e., $f$ is integrable on the real line.
To see this, notice first that
\begin{equation*}
|f(\omega)|\leq\frac{1}{(\frac{1}{2}\omega)^{2}}
\frac{4\Vert h\Vert_{L^{1}}+2\Vert h\Vert_{L^{1}}^{2}}{(1-\Vert h\Vert_{L^{1}})^{2}},
\end{equation*}
and thus $\int_{|\omega|\geq\epsilon}|f(\omega)|d\omega<\infty$
for any $\epsilon>0$.
Moreover, by L'H\^{o}pital's rule, we can check that
\begin{align*}
\lim_{\omega\rightarrow 0}f(\omega)
&=\lim_{\omega\rightarrow 0}\frac{1}{(\frac{1}{2}\omega)^{2}}
\frac{-2\Vert h\Vert_{L^{1}}+\Vert h\Vert_{L^{1}}^{2}
+\hat{h}(\omega)+\overline{\hat{h}(\omega)}
-|\hat{h}(\omega)|^{2}}{|1-\hat{h}(\omega)|^{2}}
\\
&=\frac{4}{(1-\Vert h\Vert_{L^{1}})^{2}}
\lim_{\omega\rightarrow 0}\frac{-2\Vert h\Vert_{L^{1}}+\Vert h\Vert_{L^{1}}^{2}
+\hat{h}(\omega)+\overline{\hat{h}(\omega)}
-|\hat{h}(\omega)|^{2}}{\omega^{2}}
\nonumber
\\
&=\frac{4}{(1-\Vert h\Vert_{L^{1}})^{2}}
\lim_{\omega\rightarrow 0}\frac{\hat{h}'(\omega)+\overline{\hat{h}'(\omega)}
-\hat{h}'(\omega)\overline{\hat{h}(\omega)}
-\hat{h}(\omega)\overline{\hat{h}'(\omega)}}{2\omega}.
\nonumber
\end{align*}
Since
\begin{equation*}
\hat{h}'(0)+\overline{\hat{h}'(0)}
-\hat{h}'(0)\overline{\hat{h}(0)}
-\hat{h}(0)\overline{\hat{h}'(0)}
=(1-\Vert h\Vert_{L^{1}})(\hat{h}'(0)+\overline{\hat{h}'(0)})
=0,
\end{equation*}
we apply the L'H\^{o}pital's rule again and get
\begin{align*}
\lim_{\omega\rightarrow 0}f(\omega)
&=\frac{4}{(1-\Vert h\Vert_{L^{1}})^{2}}
\lim_{\omega\rightarrow 0}\frac{\hat{h}'(\omega)+\overline{\hat{h}'(\omega)}
-\hat{h}'(\omega)\overline{\hat{h}(\omega)}
-\hat{h}(\omega)\overline{\hat{h}'(\omega)}}{2\omega}
\\
&=\frac{2}{(1-\Vert h\Vert_{L^{1}})^{2}}
\left[\hat{h}''(0)+\overline{\hat{h}''(0)}
-\hat{h}''(0)\Vert h\Vert_{L^{1}}-2\hat{h}'(0)\overline{\hat{h}'(0)}
-\overline{\hat{h}''(0)}\Vert h\Vert_{L^{1}}\right]
\nonumber
\\
&=\frac{4}{(1-\Vert h\Vert_{L^{1}})^{2}}
\left[(1-\Vert h\Vert_{L^{1}})\int_{0}^{\infty}t^{2}h(t)dt
-\left(\int_{0}^{\infty}th(t)dt\right)^{2}\right].
\nonumber
\end{align*}
Note that the first and second derivatives of $\hat{h}(\omega)$
and $\overline{\hat{h}(\omega)}$ are well defined due to the assumption
$\int_{0}^{\infty}t^{2}h(t)dt<\infty$ and are given by
\begin{align*}
&\hat{h}(\omega)=\int_{0}^{\infty}(\cos\omega t+i\sin\omega t)h(t)dt,
\\
&\overline{\hat{h}(\omega)}=\int_{0}^{\infty}(\cos\omega t-i\sin\omega t)h(t)dt,
\\
&\hat{h}'(\omega)=\int_{0}^{\infty}(-\sin\omega t+i\cos\omega t)th(t)dt,
\\
&\overline{\hat{h}'(\omega)}=\int_{0}^{\infty}(-\sin\omega t-i\cos\omega t)th(t)dt,
\\
&\hat{h}''(\omega)=\int_{0}^{\infty}(-\cos\omega t-i\sin\omega t)t^{2}h(t)dt,
\\
&\overline{\hat{h}''(\omega)}=\int_{0}^{\infty}(-\cos\omega t+i\sin\omega t)t^{2}h(t)dt.
\end{align*}
We deduce from the above that $f(\omega)\in L^{1}(\mathbb{R})$,
and then Riemann-Lebesgue theorem gives
\begin{equation*}
\lim_{t\rightarrow\infty}\int_{\mathbb{R}}e^{i\omega t}f(\omega)d\omega=0,
\end{equation*}
which implies that the limit of the real part is also zero:
\begin{equation*}
\lim_{t\rightarrow\infty}\int_{\mathbb{R}}\cos(\omega t)f(\omega)d\omega
=0.
\end{equation*}
Hence, we conclude from \eqref{eq:Kt-linear} that
\begin{align*}
&\lim_{t\rightarrow\infty}\left[K(t)-\frac{t}{(1-\Vert h\Vert_{L^{1}})^{3}}\right]
\\
&=\frac{1}{\pi(1-\Vert h\Vert_{L^{1}})^{3}}\int_{\mathbb{R}}\frac{1}{\omega^{2}}
\frac{-2\Vert h\Vert_{L^{1}}+\Vert h\Vert_{L^{1}}^{2}
+\hat{h}(\omega)+\overline{\hat{h}(\omega)}
-|\hat{h}(\omega)|^{2}}{|1-\hat{h}(\omega)|^{2}}d\omega.
\nonumber\\
&=\frac{1}{\pi(1-\Vert h\Vert_{L^{1}})^{3}}\int_{\mathbb{R}}\frac{1}{\omega^{2}}
\frac{(1-\Vert h\Vert_{L^{1}})^{2}-|1-\hat{h}(\omega)|^{2}}
{|1-\hat{h}(\omega)|^{2}}d\omega <0,
\end{align*}
where the fact that this constant is negative follows from the observation that for each $\omega,$
\begin{equation*}
|1-\hat{h}(\omega)|\geq 1-|\hat{h}(\omega)|
\geq 1-\Vert h\Vert_{L^{1}}>0.
\end{equation*}
Hence we complete the proof of Part (c).
\end{proof}

\subsection{Proof of Proposition~\ref{prop:G-property}} \label{sec:G-prop}
\begin{proof}[Proof of Proposition~\ref{prop:G-property}]
We first prove that $G$ has stationary increments. One can directly verify this fact by noting that for $\tau>0, s \ge 0,$ $G(s+ \tau)- G(s)$ is a mean zero Gaussian random variable, with variance given by
\begin{equation*}
\mbox{Var}(G(s +\tau)-G(s)) = K(s+\tau) +K(s) - 2 \mbox{Cov} (G(s+ \tau), G(s)).
\end{equation*}
Using \eqref{eq:var} and \eqref{eq:cov-G}, it is easily checked that $\mbox{Var}(G(s +\tau)-G(s))= K(\tau)= \mbox{Var}(G(\tau))$, which is independent of $s$.

We next show that the Gaussian process $G$ is not Markovian unless $h\equiv 0$.
To see this,
recall (see, e.g., Revuz and Yor \cite[p.86]{Revuz}) that a centered Gaussian process $\Upsilon$ with covariance function $\Gamma(s,t):=\mathbb{E}[\Upsilon_s \Upsilon_t]$ is Markovian if and only if
\begin{equation}\label{eq:markov-Gauss0}
\Gamma(s,u) \Gamma(t, t) = \Gamma(s,t) \Gamma(t,u),
\end{equation}
for every $0 \le s < t < u$. Given the covariance function of $G$ in \eqref{eq:cov-G}, one can directly
check that \eqref{eq:markov-Gauss0} does not hold for any nonzero exciting function $h$.

Finally, we prove that the paths of $G$ are H\"{o}lder continuous of order $\gamma$ for every $\gamma<\frac{1}{2}$. To see this, note that
$G(s+ \tau)- G(s)$ is a mean zero Gaussian random variable with variance $K(\tau)$, which implies that
for $p>0$,
\begin{equation*}
\mathbb{E}[|G(s+\tau)-G(s)|^p] = C \cdot {K(\tau)}^{p/2},
\end{equation*}
where $C=\mathbb{E}|Z|^p$ and $Z$ follows a standard normal distribution. By the Lipschitz property of $K$ in Proposition~\ref{prop:Kt}, we infer from the Kolmogorov-Chentsov theorem that the sample paths of $G$ are H\"{o}lder continuous with order less than $\frac{1}{2}$.
\end{proof}

\section{Proof of results in Section~\ref{sec:multi-hawkes}}

\subsection{Proof of Theorem~\ref{thm:multi-FCLT}}\label{sec:proof-multivariate}

\begin{proof}[Proof of Theorem~\ref{thm:multi-FCLT}]
For notational simplicity, we write {for each $t$},
\begin{equation*}
\mathbb{N} (t)=(N^{1}(t),\ldots,N^{k}(t)),
\end{equation*}
to stand for a $k-$dimensional stationary Hawkes processes where the intensity is given in \eqref{eq:intensity-multi} with $\mu=1$. It should be self-evident that here the notation $N^{i}$ stands for the $i-$th process (in Appendix A we used this notation to represent a univariate Hawkes process with baseline intensity $i$).
Let us define
\begin{equation}
\tilde{\mathbb{N}}(t):=\mathbb{N}(t)-at,
\end{equation}
and let $\tilde{\mathbb{N}}_{j}, j =1, 2, \ldots$ be independent copies of $\tilde{\mathbb{N}}$ where $\mathbb{E} [\tilde{\mathbb{N}}(t)] =0$ for each $t$.
We obtain from the immigration--birth representation of multivariate Hawkes processes that
\begin{equation}\label{eq:sum-vector}
 \mathbb{\hat N}^{(\mu)} (t) := \frac{\mathbb{N}^{(\mu)}(t)-\bar \lambda t}{\sqrt{\mu}}
=\frac{1}{\sqrt{\mu}}\sum_{j=1}^{\mu}\tilde{\mathbb{N}}_{j}(t),
\end{equation}
where as in the proof of Theorem~\ref{thm:FCLT}, it suffices to establish the weak convergence of the sequence $\left(\mathbb{\hat N}^{(\mu)} \right)$ for positive integer valued $\mu$.

We first establish the tightness of the sequence of processes $\left(\mathbb{\hat N}^{(\mu)} \right)$. We use the tightness criteria in
\cite[Chapter VI. Theorem~4.1]{Jacod2013} and verify the three conditions there. Condition (i) trivially holds. To verify Condition (ii) and (iii),
 it suffices to
check the following two conditions:
for every $0<T<\infty$, there exist some positive constants $C_{1},C_{2}$ so that
\begin{equation} \label{eq:tmp1}
\mathbb{E}\left[\Vert\tilde{\mathbb{N}}(u)-\tilde{\mathbb{N}}(s)\Vert^{2}\right]
\leq C_{1} \cdot (u-s),
\end{equation}
and
\begin{equation}\label{eq:tmp2}
\mathbb{E}\left[\Vert\tilde{\mathbb{N}}(u)-\tilde{\mathbb{N}}(t)\Vert^{2}\Vert\tilde{\mathbb{N}}(t)-\tilde{\mathbb{N}}(s)\Vert^{2}\right]
\leq C_{2} \cdot (u-s)^{2},
\end{equation}
for all $0\leq s\leq t\leq u\leq T$ with $u-s<1$, where the notation $|| \cdot ||$ stands for the usual Euclidean norm of a vector in $\mathbb{R}^k.
$ To see this, first notice that using Markov inequality, it is straightforward to verify that \eqref{eq:tmp1} implies Condition (ii) in \cite[Chapter VI. Theorem~4.1]{Jacod2013}. In addition, following the proof of Theorem~2 in \cite{Hahn} (the processes considered there are real--valued, but the argument in that proof also works for $\mathbb{R}^k$-valued processes), one can immediately deduce that \eqref{eq:tmp2} implies Condition (iii) in \cite[Chapter VI. Theorem~4.1]{Jacod2013}.

%
%
%
%
%

We now prove \eqref{eq:tmp1} and \eqref{eq:tmp2}. As the dimension $k$ of the multivariate Hawkes process $\mathbb{N}$ is finite, in order to prove \eqref{eq:tmp1} and \eqref{eq:tmp2},
it suffices to check that
for every $0<T<\infty$, there exist some positive constants $C_{1},C_{2}$ so that for all $0\leq s\leq t\leq u\leq T$ with $u-s<1$
and every $1\leq i,j\leq k$,
\begin{equation} \label{eq:multi1}
\mathbb{E}\left[(N^{i}(s,u])^{2}\right]
\leq C_{1}\cdot (u-s),
\end{equation}
and
\begin{equation}\label{eq:multi2}
\mathbb{E}\left[(N^{i}(t,u])^{2}(N^{j}(s,t])^{2}\right]
\leq C_{2} \cdot (u-s)^{2}.
\end{equation}

We next prove \eqref{eq:multi1} and \eqref{eq:multi2}.
Similar as \eqref{C1estimate}, we can compute that
\begin{align*}
\mathbb{E}\left[(N^{i}(s,u])^{2}\right]
&\leq 2\mathbb{E}\left[\left(N^{i}(s,u]-\int_{s}^{u}\lambda^{i}(v)dv\right)^{2}\right]
+2\mathbb{E}\left[\left(\int_{s}^{u}\lambda^{i}(v)dv\right)^{2}\right]
\\
&=2a_{i}(u-s)
+2\mathbb{E}\left[\left(\int_{s}^{u}\lambda^{i}(v)dv\right)^{2}\right]
\\
&\leq
2a_{i}(u-s)
+2(u-s)\mathbb{E}\left[\left(\int_{s}^{u}(\lambda^{i}(v))^{2}dv\right)\right]
\\
&=2a_{i}(u-s)
+2(u-s)^{2}\mathbb{E}[(\lambda^{i}(0))^{2}]\leq C_{1}(u-s),
\end{align*}
for some positive constant $C_{1}$, \textit{provided that $\mathbb{E}[(\lambda^{i}(0))^{2}]<\infty$.}

Moreover, similar as the derivations in \eqref{eq:1}, \eqref{eq:221} and \eqref{22}, we have
\begin{align*}
&\mathbb{E}\left[(N^{i}(t,u])^{2}(N^{j}(s,t])^{2}\right]
\\
&\leq
2\mathbb{E}\left[\left(N^{i}(t,u]-\int_{t}^{u}\lambda^{i}(v)dv\right)^{2}(N^{j}(s,t])^{2}\right]
+2\mathbb{E}\left[\left(\int_{t}^{u}\lambda^{i}(v)dv\right)^{2}(N^{j}(s,t])^{2}\right]
\\
&=
2\mathbb{E}\left[\int_{t}^{u}\lambda^{i}(v)dv\cdot (N^{j}(s,t])^{2}\right]
+2\mathbb{E}\left[\left(\int_{t}^{u}\lambda^{i}(v)dv\right)^{2}(N^{j}(s,t])^{2}\right]
\\
&\leq
2\left(\mathbb{E}\left[\left(\int_{t}^{u}\lambda^{i}(v)dv\right)^{2}\right]\right)^{1/2}
\left(\mathbb{E}\left[(N^{j}(s,t])^{4}\right]\right)^{1/2}
+2(u-t)\mathbb{E}\left[\left(\int_{t}^{u}(\lambda^{i}(v))^{2}dv\right)(N^{j}(s,t])^{2}\right]
\\
&\leq
2(u-t)^{1/2}\left(\int_{t}^{u}\mathbb{E}(\lambda^{i}(v))^{2}dv\right)^{1/2}
\mathbb{E}\left[(N^{j}(s,t])^{4}\right]^{1/2}
+(u-t)\mathbb{E}\left[\int_{t}^{u}\left((\lambda^{i}(v))^{4}+(N^{j}(s,t])^{2}\right)dv\right]
\\
&=
2(u-t)\left(\mathbb{E}\left[(\lambda^{i}(0))^{2}\right]\right)^{1/2}
\left(\mathbb{E}\left[(N^{j}(0,t-s])^{4}\right]\right)^{1/2}
+(u-t)^{2}\left(\mathbb{E}\left[(\lambda^{i}(0))^{4}\right]+\mathbb{E}\left[(N^{j}(0,t-s])^{2}\right]\right)
\\
&\leq
2(u-t)\left(\mathbb{E}\left[(\lambda^{i}(0))^{2}\right]\right)^{1/2}
\left(\mathbb{E}\left[(N^{j}(0,t-s])^{4}\right]\right)^{1/2}
+(u-t)^{2}\left(\mathbb{E}\left[(\lambda^{i}(0))^{4}\right]+C_{1}(t-s)\right).
\end{align*}
Similar as \eqref{eq:2} in the proof of Theorem \ref{thm:FCLT}, we have
\begin{equation*}
\mathbb{E}\left[(N^{j}(0,t-s])^{4}\right]
\leq
8\bar{C}(t-s)^{2}\mathbb{E}[(\lambda^{j}(0))^{2}]
+8(t-s)^{4}\mathbb{E}[(\lambda^{j}(0))^{4}].
\end{equation*}
Hence, we obtain
\begin{equation*}
\mathbb{E}\left[(N^{i}(t,u])^{2}(N^{j}(s,t])^{2}\right]\leq C_{2}(u-s)^{2},
\end{equation*}
for some positive constant $C_{2}$,
\textit{provided that} $\mathbb{E}[(\lambda^{j}(0))^{4}]<\infty$ for $1 \le j \le k$.

It remains to prove that $\mathbb{E}[(\lambda^{j}(0))^{4}]<\infty$ for $1 \le j \le k$, as it implies $\mathbb{E}[(\lambda^{j}(0))^{2}]<\infty$.
Similar as the derivations in \eqref{eq:tempstep}--\eqref{eq:3} in the proof of Theorem \ref{thm:FCLT},
since $h_{ij}$ are locally bounded and Riemann integrable by Assumption~1,
for sufficiently small $\delta>0$, we obtain
\begin{align*}
\mathbb{E}[(\lambda^{j}(0))^{4}]
&=\mathbb{E}\left[\left(p_{j}+\sum_{\ell=1}^{k}\int_{-\infty}^{0-}h_{j\ell}(-s)N^{\ell}(ds)\right)^{4}\right]
\\
&\leq C \left[(p_{j})^{4}
+\sum_{\ell=1}^{k}\mathbb{E}\left[\left(\int_{-\infty}^{0-}h_{j\ell}(-s)N^{\ell}(ds)\right)^{4}\right]\right]
\\
&\leq
C \left[(p_{j})^{4}
+\sum_{\ell=1}^{k}\left(\sum_{i=0}^{\infty}\max_{t\in[-(i+1)\delta,-i\delta]}h_{j\ell}(t)\right)^{4}
\mathbb{E}\left[\left(N^{\ell}(0,\delta]\right)^{4}\right]\right],
\end{align*}
for some positive constant $C.$
Thus, it remains to show that for every $1\leq\ell\leq k$, $\mathbb{E}\left[\left(N^{\ell}(0,1]\right)^{4}\right]<\infty$.
It suffices to show that there exists some constant $c_{\ell}>0$
so that $\mathbb{E}[e^{c_{\ell}N^{\ell}(0,1]}]<\infty$.
Let us define $\mathbb{E}^{\emptyset}$ as the expectation under
which the process $\mathbb{N}=(N^{1},\ldots,N^{k})$ (with slight abuse of notations) is a multivariate Hawkes process starting
from empty history, that is, $\lambda^{i}(t)=p_{i}+\sum_{j=1}^{k}\int_{0}^{t-}h_{ij}(t-s)N^{j}(ds)$.
For any $\psi_{i}>0$,
$e^{\sum_{i=1}^{k}\psi_{i}N^{i}(t)-\sum_{i=1}^{k}(e^{\psi_{i}}-1)\int_{0}^{t}\lambda^{i}(s)ds}$
is a martingale (see e.g. \cite{SH}), and thus
\begin{align*}
1&=\mathbb{E}^{\emptyset}\left[e^{\sum_{i=1}^{k}\psi_{i}N^{i}(t)-\sum_{i=1}^{k}(e^{\psi_{i}}-1)\int_{0}^{t}\lambda^{i}(s)ds}\right]
\\
&=\mathbb{E}^{\emptyset}\left[e^{\sum_{i=1}^{k}\psi_{i}N^{i}(t)
-\sum_{i=1}^{k} (e^{\psi_{i}}-1) \int_{0}^{t}(p_{i}+\int_{0}^{s}\sum_{j=1}^{k}h_{ij}(s-u)N^{j}(du))ds}
\right]
\\
&\geq
\mathbb{E}^{\emptyset}\left[e^{\sum_{i=1}^{k}\psi_{i}N^{i}(t)-\sum_{i=1}^{k}(e^{\psi_{i}}-1)(p_{i}t+\sum_{j=1}^{k}\Vert h_{ij}\Vert_{L^{1}}N^{j}(t))}\right],
\end{align*}
which implies that
\begin{equation}\label{eq:tempo1}
\mathbb{E}^{\emptyset}\left[e^{\sum_{i=1}^{k}(\psi_{i}-\sum_{j=1}^{k}(e^{\psi_{j}}-1)\Vert h_{ji}\Vert_{L^{1}})N^{i}(t)}\right]
\leq e^{\sum_{i=1}^{k}(e^{\psi_{i}}-1)p_{i}t}.
\end{equation}
Since the spectral radius of the matrix $\mathbb{H}= (\Vert h_{ij}\Vert_{L^{1}})_{1 \le i,j \le k}$ is strictly less than $1$,
we know that $(\mathbb{I}-\mathbb{H})^{-1}$ exists and
$(\mathbb{I}-\mathbb{H})^{-1}=\sum_{n=0}^{\infty}\mathbb{H}^{n}$.
Thus for any fixed positive column vector $m=(m_{1},\ldots,m_{k})^{t}\in\mathbb{R}_{>0}^{k}$,
we have $((\mathbb{I}-\mathbb{H})^{-1}m)_{i}>0$ for every $i$,
where $((\mathbb{I}-\mathbb{H})^{-1}m)_{i}$ is the $i$-th component of the vector
$(\mathbb{I}-\mathbb{H})^{-1}m$.
Let $\psi_{i}=\epsilon((\mathbb{I}-\mathbb{H})^{-1}m)_{i}$,
where $\epsilon>0$ is sufficiently small so that
we can find some constant $C(m)$ that depends on $m$ such that
\begin{equation*}
\psi_{i}-\sum_{j=1}^{k}(e^{\psi_{j}}-1)\mathbb{H}_{ji}
\geq\psi_{i}-\sum_{j=1}^{k}\psi_{j}\mathbb{H}_{ji}-C(m)\epsilon^{2}
=\epsilon m_{i}-C(m)\epsilon^{2}>0.
\end{equation*}
Hence, we deduce from \eqref{eq:tempo1} that there exists $c_{i}=\epsilon m_{i}-C(m)\epsilon^{2}>0$ so that
\begin{equation*}
\mathbb{E}^{\emptyset}\left[e^{\sum_{i=1}^{k}c_{i}N^{i}(t)}\right]
\leq e^{t\sum_{i=1}^{k}p_{i}(e^{\epsilon((\mathbb{I}-\mathbb{H})^{-1}m)_{i}}-1)}.
\end{equation*}
In particular, for every $1\leq\ell\leq k$, we have
\begin{equation*}
\mathbb{E}^{\emptyset}\left[e^{c_{\ell}N^{\ell}(t)}\right]
\leq e^{t\sum_{i=1}^{k}p_{i}(e^{\epsilon((\mathbb{I}-\mathbb{H})^{-1}m)_{i}}-1)}.
\end{equation*}
Since the linear Hawkes process (either with empty history or the stationary version) is associated,
we then deduce that for positive integer $t$,
\begin{equation*}
\prod_{n=1}^{t}\mathbb{E}^{\emptyset}\left[e^{c_{\ell}N^{\ell}(n-1,n]}\right]
\leq\mathbb{E}^{\emptyset}\left[e^{c_{\ell}N^{\ell}(t)}\right]
\leq e^{t\sum_{i=1}^{k}p_{i}(e^{\epsilon((\mathbb{I}-\mathbb{H})^{-1}m)_{i}}-1)}.
\end{equation*}
Hence, by the ergodicity of the Hawkes processes with empty history where the exciting function satisfies Assumption~1, see e.g. \cite{Bremaud}, we obtain
\begin{equation*}
\log\mathbb{E}\left[e^{c_{\ell}N^{\ell}(0,1)}\right]
=\lim_{t\rightarrow\infty}\frac{1}{t}\sum_{n=1}^{t}\log\mathbb{E}^{\emptyset}\left[e^{c_{\ell}N^{\ell}(n-1,n]}\right]
\leq\sum_{i=1}^{k}p_{i}(e^{\epsilon((I-\mathbb{H}^{t})^{-1}m)_{i}}-1)<\infty,
\end{equation*}
where we recall $\mathbb{E}$ is the expectation under which the Hawkes process is stationary.
Hence, we have proved that there exists some constant $c_{\ell}>0$
so that $\mathbb{E}[e^{c_{\ell}N^{\ell}(0,1]}]<\infty$ for each $1 \le l \le k$.

Now we have proved the tightness of the sequence $\left(\mathbb{\hat N}^{(\mu)} \right)$, we next
 show that the finite dimensional distributions of the sequence of processes $\left(\mathbb{\hat N}^{(\mu)} \right)$ converges in distribution to that of the limiting process $\mathbb{G}$ as $\mu \rightarrow \infty$. To this end, we note that one can readily compute the covariance function of $\tilde{\mathbb{N}}$ as in the univariate case (see Equations. \eqref{CovCompute1}-\eqref{CovCompute4}), and find that
for $t\geq s$,
\begin{equation*}
\text{Cov}(\tilde{\mathbb{N}} (t),\tilde{\mathbb{N}}(s))
=\int_{s}^{t}\int_{0}^{s}\Phi(u-v)dudv+\mathbb{K}(s) = \text{Cov}(\mathbb{G}(t),\mathbb{G}(s)).
\end{equation*}
Hence, in view of \eqref{eq:sum-vector} and \eqref{eq:tmp1}, the weak convergence of the finite dimensional distributions of this sequence $\left(\mathbb{\hat N}^{(\mu)} \right)$ immediately follows from the central limit theorem for sum of i.i.d. random vectors and the Cram\'{e}r-Wold device (see e.g. Section~4.3.2 in \cite{whitt2002}).

Finally, as the covariance of $\mathbb{G}$ is the same as that of $\tilde{\mathbb{N}}$, we immediately infer from \eqref{eq:tmp1} that for $s<u,$
\begin{equation*}
\mathbb{E}\left[\Vert \mathbb{G}(u)-\mathbb{G}(s)\Vert^{2}\right]
\leq C_{1} \cdot (u-s),
\end{equation*}
which implies that the limiting Gaussian process $\mathbb{G}$ has continuous sample paths (\cite[p.37]{Revuz}).
The proof is therefore complete.
\end{proof}

\subsection{Proof of Proposition~\ref{thm:IS-multi-gen}}\label{sec:proof-13}

\begin{proof}[Proof of Proposition~\ref{thm:IS-multi-gen}]
To prove Proposition~\ref{thm:IS-multi-gen}, we rely on \cite{Kri1997}. For notational simplicity, we prove the joint weak convergence of $(\mathbb{X}_1^{\mu}, \ldots, \mathbb{X}_k^{\mu})$ as $\mu \rightarrow \infty$ for the case $k=2.$ The general case $k \ge 2$ follows similarly.

First, as in the proof of Theorem 3 in \cite{Kri1997}, we obtain that for $i=1, 2$,
\begin{eqnarray*}
\mathbb{X}_i^{\mu}(t) &=&{\sqrt{\mu}} \left( \frac{\mathbb{Q}_i^{\mu}(t)}{\mu} - q_{i0} (1-F_{i0}(t)) - a_i \cdot \int_{0}^t (1-F_i(t-u)) du \right), \\
&=& \mu^{-1/2} \sum_{j=1}^{\mathbb{Q}^{\mu}_i(0)} (1_{\bar \eta_{i,j} >t} - (1- F_{i0}(t)))  + (1 - F_{i0}(t)) \mu^{1/2}(\mu^{-1}\mathbb{Q}^{\mu}_i(0) - q_{i0}) \\
&&+[M^{\mu}_{i1}(t) -M^{\mu}_{i2}(t)],
\end{eqnarray*}
where for $t \ge 0,$
\begin{align*}
M^{\mu}_{i1}(t) &= \int_{0}^t (1 - F_i(t-s)) d \left[ \frac{\mathbb{N}_i^{(\mu)}(s)-\bar \lambda_i s}{\sqrt{\mu}} \right],\\
M^{\mu}_{i2}(t) &= \int_{0}^t \int_{0}^t 1_{s+x \le t} d U_i^\mu \left( \frac{\mathbb{N}_i^{(\mu)}(s)}{{\mu}} , F_i(x)\right),\\
U_i^\mu (t, x)  &=\mu^{-1/2} \sum_{j=1}^{\lfloor \mu t \rfloor} (1_{\zeta_{ij} \le x} -x).
\end{align*}
Here $\zeta_{ij}$ are all independent and uniformly distributed random variables on $[0, 1]$ and the service times $\eta_{ij}=F_i^{-1}(\zeta_{ij})$, where
$F_i^{-1}(x):= \inf\{y: F_i(y) \ge x \}$.
For each fixed $i$, it was proved in \cite{Kri1997} (see (6.1)--(6.3) there) that the following weak convergence of processes hold:
\begin{eqnarray}
&\mu^{-1/2} \sum_{j=1}^{\mathbb{Q}^{\mu}_i(0)} (1_{\bar \eta_{i,j} >t} - (1- F_{i0}(t)))  \Rightarrow \sqrt{q_{i0}} W^{i0}(F_{i0}(t)), \label{eq:cov1}\\
&(1 - F_{i0}(t)) \mu^{1/2}(\mu^{-1}\mathbb{Q}^{\mu}_i(0) - q_{i0})   \Rightarrow (1 - F_{i0}(t)) \xi_i, \label{eq:cov2} \\
& \left(M^{\mu}_{i1}(t), M^{\mu}_{i2}(t) \right)  \Rightarrow  \left( \int_{0}^t  (1- F_i(t-u)) d\mathbb{G}_i(u), \int_0^t \int_0^t  1_{s +x \le t} dU_i \left(a_i s, F_i(x) \right)  \right). \label{eq:cov3}
\end{eqnarray}
In addition, there is clearly a joint weak convergence of the left-hand sides of \eqref{eq:cov1}--\eqref{eq:cov3} to the right-hand sides \cite{Kri1997}. Now, by the hypothesis, the number of customers in the system at time zero, $\mathbb{Q}^{\mu}_i(0)$ for $i=1, 2$, as well as their respective service requirements $\bar \eta_{ij}$ for $i=1, 2,$ are mutually independent. Moreover, the arrivals of new customers and the service requirements of those new customers are independent of the initial number of customers $\mathbb{Q}^{\mu}(0)$ and their service times.
Hence, in order to prove the joint weak convergence of $(\mathbb{X}_1^{\mu}, \mathbb{X}_2^{\mu})$, it suffices to prove the weak convergence of $(M^{\mu}_{11}, M^{\mu}_{12}, M^{\mu}_{21}, M^{\mu}_{22})$ as $\mu \rightarrow \infty.$

To this end, let us define
\begin{align*} \label{eq:M-mu-2}
\tilde{M}^{\mu}_{i2} = \int_{0}^t \int_{0}^t 1_{s+x \le t} d U_i^\mu \left( a_i s, F_i(x)\right), \quad \text{for $i=1, 2$. }
\end{align*}
Note that by Theorem~\ref{thm:multi-FCLT}, we have the sequence of processes $\left(\frac{\mathbb{N}^{(\mu)}}{\mu}\right)$ converges in distribution to a deterministic limit process $\omega$ where $\omega(t): =at$ for each $t \ge 0.$ As $\omega$ has continuous paths and the Skorohod $J_1$ topology relativized to the space of continuous functions coincides with the uniform topology there (\cite[p.124]{Billingsley}), we obtain that
for each $T>0$, as $\mu \rightarrow \infty$,
\begin{equation*}
\sup_{0 \le t \le T} ||{\mathbb{N}^{(\mu)}(t)}/{\mu} - at ||  \rightarrow 0 \quad \text{in probability}.
\end{equation*}
Then using a similar argument as in the proof of Lemma~5.3 in \cite{Kri1997}, we can establish that for each $T>0$ and $\epsilon>0$,
\begin{equation} \label{eq:convg-p}
\lim_{\mu \rightarrow \infty} P\left( \sup_{t \le T} |\tilde{M}^{\mu}_{i2}(t) - {M}^{\mu}_{i2} (t)|> \epsilon\right) =0, \quad \text{for $i=1, 2$. }
\end{equation}

In addition, using integration by parts, we can write $(M^{\mu}_{11}, M^{\mu}_{21}) = \left( g_1( \mathbb{\hat N}_1^{(\mu)} ), g_2 (\mathbb{ \hat N}_2^{(\mu)} ) \right) $, where $ \mathbb{\hat N}_i^{(\mu)} (s) := \frac{\mathbb{N}_i^{(\mu)}(s)-\bar \lambda_i s}{\sqrt{\mu}}$ for each $s \ge 0$,
$g_i: D([0,\infty),\mathbb{R}) \rightarrow D([0,\infty),\mathbb{R}),$ is defined by
\begin{equation*}
g_i(x(\cdot)) (t) = x(t) - \int_{0}^t x(t-s) dF_i(s), \quad \text{for $i=1, 2,$}
\end{equation*}
and $g_i$ is continuous at points $x(\cdot) \in C([0,\infty),\mathbb{R})$.
See the proof of Lemma~3.3 in \cite{Kri1997}. Now in Theorem~\ref{thm:multi-FCLT} we have established that $\left(\mathbb{\hat N}_1^{(\mu)}, \mathbb{\hat N}_2^{(\mu)} \right)$ converges in distribution to the Gaussian process $(\mathbb{G}_1, \mathbb{G}_2)$ under the Skorohod $J_1$ topology where the limiting Gaussian process has continuous paths, it then immediately follows that
\begin{equation}\label{eq:convg1}
(M^{\mu}_{11}, M^{\mu}_{21}) \Rightarrow (M_{11}, M_{21}), \quad \text{as $\mu \rightarrow \infty$,}
\end{equation}
where
\begin{equation*}
M_{i1}  (t) = \int_{0}^t (1 - F_i(t-s)) d \mathbb{G}_i(s), \quad \text{for $i=1, 2$. }
\end{equation*}


Furthermore, as the service processes of each class $i$ customers are independent, we deduce that the two processes $U_1^\mu$ and $U_2^\mu$ are independent for each $\mu$, which further implies that $\tilde M^{\mu}_{12}$ and $\tilde M^{\mu}_{22}$ are two independent processes.
By Lemma~3.1 of \cite{Kri1997}, we have $U_i^\mu \Rightarrow U_i$ in $D([0,\infty),D[0,1])$ for each $i$ as $\mu \rightarrow \infty$. Hence, we deduce from Lemma~5.3 of \cite{Kri1997} that
\begin{equation} \label{eq:convg2}
(\tilde M^{\mu}_{12}, \tilde M^{\mu}_{22}) \Rightarrow (M_{12}, M_{22}), \quad \text{as $\mu \rightarrow \infty$,}
\end{equation}
where
\begin{equation*}
M_{i2}(t) = \int_{0}^t \int_{0}^t 1_{s+x \le t} d U_i \left( a_i t, F_i(x)\right), \quad \text{for $i=1, 2$. }
\end{equation*}
Then we can obtain from \eqref{eq:convg1}, \eqref{eq:convg2} and the independence of the service processes and the arrival processes of each class of customers that
\begin{equation*}
(M^{\mu}_{11}, \tilde M^{\mu}_{12}, M^{\mu}_{21}, \tilde M^{\mu}_{22}) \Rightarrow (M_{11}, M_{12}, M_{21}, M_{22}), \quad \text{as $\mu \rightarrow \infty$.}
\end{equation*}
Together with \eqref{eq:convg-p} which implies that $\left({M}^{\mu}_{12} - \tilde{M}^{\mu}_{12}, {M}^{\mu}_{22} - \tilde{M}^{\mu}_{22}  \right) \Rightarrow (0, 0)$, we infer that the process $(M^{\mu}_{11}, M^{\mu}_{12}, M^{\mu}_{21}, M^{\mu}_{22})$ converges in distribution to $(M_{11}, M_{12}, M_{21}, M_{22})$ as $\mu \rightarrow \infty$. Therefore, we obtain the weak convergence of $(\mathbb{X}_1^{\mu}, \mathbb{X}_2^{\mu})$ to the desired limit process $(\mathbb{X}_1, \mathbb{X}_2)$. The proof is completed.
\end{proof}

\end{document}